\documentclass[12pt,a4paper,leqno]{amsart}

\usepackage[latin1]{inputenc}
\usepackage[T1]{fontenc}
\usepackage{amsfonts}
\usepackage{amsmath}
\usepackage{amssymb}
\usepackage{eurosym}
\usepackage{mathrsfs}
\usepackage{palatino}
\usepackage{color}
\usepackage{esint}
\usepackage{url}
 \usepackage{color}

 \usepackage{amsrefs}     
 \usepackage[normalem]{ulem}
 \usepackage{hyperref} 
\hypersetup{
    colorlinks=false,       
    linkcolor=blue,          
    citecolor=magenta,        
    filecolor=magenta,      
    urlcolor=cyan           
}

\newcommand{\R}{\mathbb{R}}
\newcommand{\C}{\mathbb{C}}

\newcommand{\N}{\mathbb{N}}

\newcommand{\Z}{\mathbb{Z}}

\numberwithin{equation}{section}
\usepackage{graphicx}

\swapnumbers
\theoremstyle{plain}
\newtheorem{thm}[equation]{Theorem}
\newtheorem{lem}[equation]{Lemma}
\newtheorem{prop}[equation]{Proposition}
\newtheorem{cor}[equation]{Corollary}

\theoremstyle{definition}

\theoremstyle{remark}
\newtheorem{rem}[equation]{Remark}

\author{Michael T. Lacey}
\address[M.T.L.]{ School of Mathematics, Georgia Institute of Technology, Atlanta GA 30332, USA}
\email {lacey@math.gatech.edu}
\thanks{Research of M.T.L. is supported in part by grant NSF-DMS 0968499, and the Australian Research Council through grant ARC-DP120100399.}

\author{Henri Martikainen}
\address[H.M.]{Department of Mathematics and Statistics, University of Helsinki, P.O.B. 68, FI-00014 Helsinki, Finland}
\email{henri.martikainen@helsinki.fi}
\thanks{Research of H.M. is supported by the Academy of Finland through the grant
Multiparameter dyadic harmonic analysis and probabilistic methods.}

\makeatletter
\@namedef{subjclassname@2010}{%
  \textup{2010} Mathematics Subject Classification}
\makeatother


\title[Local $Tb$ theorem with $L^2$ testing conditions and general measures]{Local $Tb$ theorem with $L^2$ testing conditions and general measures: Calder\'on--Zygmund operators}

\begin{document}
\begin{abstract}
Local $Tb$ theorems with $L^p$ type testing conditions, which are not scale invariant, have been studied widely in the case of the Lebesgue measure. Until very recently,
local $Tb$ theorems in the non-homogeneous case had only been proved assuming scale invariant ($L^{\infty}$ or BMO) testing conditions. The combination of non-scale-invariance and general measures
is a delicate issue. In a previous paper
we overcame this obstacle in the model case of square functions defined using general measures. In this paper we finally tackle the very demanding case of Calder\'on--Zygmund operators.
That is, we prove a non-homogeneous local $Tb$ theorem with $L^2$ type testing conditions for all Calder\'on--Zygmund operators. In doing so we prove general twisted martingale transform
inequalities which turn out to be subtle in our general framework.
\end{abstract}

\maketitle

\section{Introduction}
In this paper we prove the boundedness of a Calder\'on--Zygmund operator $T$ on $L^2(\mu)$, where $\mu$ can be non-homogeneous,
assuming only local and non-scale-invariant testing conditions. While such local $Tb$ theorems with $L^p$ testing functions are known in the
homogeneous case, proving such a result in the non-homogeneous setting is delicate. He we are able to do this for the first time. The proof requires extensive development and
usage of the techniques of non-homogeneous and two-weight dyadic analysis.

Let us begin by introducing the setting and formulate our main theorem. We assume that $\mu$ is a measure on $\R^n$ satisfying only the size condition $\mu(B(x,r)) \lesssim r^m$ for some $m$.
We consider Calder\'on--Zygmund operators $T$ in this setting. First of all, this means that there is a kernel $K \colon \R^n \times \R^n \setminus \{(x,y): x = y\} \to \C$ for which
there holds for some $C < \infty$ and $\alpha > 0$ that
\begin{displaymath}
|K(x,y)| \le \frac{C}{|x-y|^m}, \qquad x \ne y,
\end{displaymath}
\begin{displaymath}
|K(x,y) - K(x',y)| \le C\frac{|x-x'|^{\alpha}}{|x-y|^{m+\alpha}}, \qquad |x-y| \ge 2|x-x'|,
\end{displaymath}
and
\begin{displaymath}
|K(x,y) - K(x,y')| \le C\frac{|y-y'|^{\alpha}}{|x-y|^{m+\alpha}}, \qquad |x-y| \ge 2|y-y'|.
\end{displaymath}
Secondly, we demand that $T$ is a linear operator satisfying the identity
\begin{displaymath}
Tf(x) = \int_{\R^n} K(x,y)f(y)\,d\mu(y), \qquad x \not \in \textup{spt}\,f.
\end{displaymath}
In this paper we assume a priori that $T\colon L^2(\mu) \to L^2(\mu)$ boundedly. We are after a new quantitative bound for $\|T\|$, independent of the a priori bound. Such practice is standard, and one
can deduce to this situation by, for example, considering suitably truncated operators.

We are ready to state our main theorem -- a non-homogeneous local $Tb$ theorem with $L^2$ type testing conditions for all Calder\'on--Zygmund operators.
\begin{thm}\label{thm:main}
Suppose that $T\colon L^2(\mu) \to L^2(\mu)$ is a bounded Calder\'on--Zygmund operator with an adjoint operator $T^*$. We assume that to every cube $Q \subset \R^n$ there is associated two functions
$b^T_Q$ and $b^{T^*}_Q$ satisfying that
\begin{enumerate}
\item spt$\,b^T_Q \subset Q$ and spt$\,b^{T^*}_Q \subset Q$;
\item \begin{displaymath}
\Big| \int_Q b^T_Q\,d\mu \Big| \gtrsim \mu(Q) \qquad \textup{and} \qquad  \Big| \int_Q b^{T^*}_Q\,d\mu \Big| \gtrsim \mu(Q);
\end{displaymath}
\item $\| b^T_Q \|_{L^2(\mu)}^2 \lesssim \mu(Q)$ and $\| b^{T^*}_Q \|_{L^2(\mu)}^2 \lesssim \mu(Q)$;
\item $\|1_QTb^T_Q \|_{L^2(\mu)}^2 \lesssim \mu(Q)$ and $\|1_QT^*b^{T^*}_Q \|_{L^2(\mu)}^2 \lesssim \mu(Q)$.
\end{enumerate}
Then we have that $\|T\| \lesssim 1$.
\end{thm}

Recently in \cite{LM:SF} we proved a version of this theorem for square functions defined in the upper half-space. While of independent interest because of the genuinely different context, it is a result with a much simpler proof than the current one.
Indeed, the square functions essentially provide a model framework where many technicalities of the Calder\'on--Zygmund world do not arise. One of them is that the diagonal is completely trivial for square functions while
extremely delicate for Calder\'on--Zygmund operators. Another difference is that the recent Whitney averaging identity over good cubes of Martikainen and Mourgoglou \cite{MM} makes certain probabilistic arguments easy even in the local $Tb$ situation.
A critical difference is the fact that the paraproduct operator is much simpler in the square function case.

Before going more to the history and context, we want to discuss the proof of our main theorem,  Theorem \ref{thm:main}, and the references most related to our techniques. The proof is
quite simply begun by reducing to a non-homogeneous $T1$ theorem
of Nazarov--Treil--Volberg \cite{NTV1}. More specifically, a local formulation following directly from this is used:
\begin{displaymath}
\|T\| \le C_1 + C_2\mathop{\sup_{Q_0 \subset \R^n}}_{Q_0 \textup{ cube}} \mathop{\sup_{f, g}}_{|f|, |g| \le 1_{Q_0}} \frac{|\langle Tf, g\rangle|}{\mu(\lambda Q_0)}.
\end{displaymath}
Here $\lambda > 1$ is some fixed large constant. This reduces things to proving that
\begin{displaymath}
|\langle Tf, g\rangle| \le (C_3 + c\|T||)\mu(\lambda Q_0),
\end{displaymath}
where $c$ can be taken to be arbitrarily small. Two independent random cubes $Q^*$ and $R^*$ for which $Q_0 \subset Q^* \subset \lambda Q_0$ and $Q_0 \subset R^* \subset \lambda Q_0$ are then used to expand the fixed bounded functions
$f$ and $g$ dyadically in to martingale differences adapted to the local test functions.

We now come to the essentials. To handle the complicated paraproducts we require a non-homogeneous version of the twisted martingale difference inequalities of Auscher--Routin
\cite{AR} or Lacey--V\"ah\"akangas \cite{LV-perfect}. This is Proposition \ref{prop:twisted} of our current paper
-- a result of independent interest. Indeed, the proof of Proposition \ref{prop:twisted} turns out to be a demanding task. The key reason lies in the fact that even if we have performed a stopping time argument which gives us that a fixed test function $b^T_F$ behaves nicely on a cube $Q$ i.e. $\int_Q |b^T_F|^2\, d\mu \lesssim \mu(Q)$, we cannot say much what happens in the stopping children of $Q$. That is, in a stopping child $Q'$ of $Q$ we cannot use the simple argument
\begin{displaymath}
\int_{Q'} |b^T_F|^2 \,d\mu \le \int_{Q} |b^T_F|^2 \,d\mu \lesssim \mu(Q) \lesssim \mu(Q')
\end{displaymath}
which would only available if $\mu$ would be doubling.

Instead, the proof of Proposition \ref{prop:twisted} becomes about controlling \emph{maximal truncations} of certain \emph{half-twisted} martingales $\sum_Q \epsilon_Q D_Q$. Even if we are interested in an $L^2$ result,
we find it convenient to prove a weak type bound for every $p \in (1,\infty)$ and interpolate this (the half-twisted martingales will be $L^p$ bounded for every $p$ unlike the original twisted martingales).
But such a weak type bound can be reduced to a testing condition -- an idea originally by Sawyer \cite{Sa}, 
but which can essentially also be found from e.g. \cite{seven} by Hyt\"onen et al. The verification of this testing inequality is based crucially on controlling $\sum_Q \epsilon_Q D_Q 1$ in $L^p$.
This control is proved by reducing to the case $p=1$ using a non-homogeneous John--Nirenberg principle formulated at least by Lacey--Petermichl--Reguera \cite{LPR} and Hyt\"onen--P\'erez--Treil--Volberg \cite{HCTV}.

Proposition \ref{prop:twisted} is formulated in such a way that essentially the stopping generation is fixed. For this reason we perform an argument which gives that in the expansion
of the pairing $\langle Tf , g\rangle$ we can use only finitely many generations of stopping cubes. This follows from the Carleson property of the stopping cubes by noticing that the large generations provide
only an absorbable error. 

After this, the pairing is split in to standard pieces: separated, nested, diagonal. The goodness is inserted only to the nested sum -- an idea already used by Hyt\"onen--Martikainen \cite{HM}. 
The point of adding the goodness like this is to guarantee the collapse of the paraproduct.
The crucial thing is that
the paraproduct arising from this sum can now be handled using an argument by Lacey--V\"ah\"akangas \cite{LV-direct}, the non-homogeneous
twisted martingale difference inequality proved before, and the reduction to finitely many generations of stopping cubes.
The last thing is to deal in this non-homogeneous setting with the extremely delicate surgery of the diagonal using only $L^2$ test functions.

We have given the technical foundation and references related to these latest techniques. But let us now discuss the history and overall context of the problem.
The first local $Tb$ theorem, with $L^{\infty}$ control of the test functions and their images, is by Christ \cite{Ch}. This was proven for doubling measures.
Nazarov, Treil and Volberg \cite{NTVa} obtained a non-homogeneous version of this theorem.

The idea of using (in the homogeneous situation) just local $L^p$ type testing conditions was introduced by Auscher, Hofmann, Muscalu, Tao and Thiele \cite{AHMTT}.
However, their proof works only for the so-called perfect dyadic singular integral operators.
The assumptions are of the form $\int_Q |b^1_Q|^p \le |Q|$, $\int_Q |b^2_Q|^q \le |Q|$, $\int_Q |Tb^1_Q|^{q'} \le |Q|$ and $\int_Q |T^*b^2_Q|^{p'} \le |Q|$, where $s'$ denotes the dual exponent of $s$ and $1 < p, q \le \infty$.
Our interest here is the case $p=q=2$  for all Calder\'on--Zygmund operators in the non-homogeneous setting. Even in the homogeneous setting extending the result of \cite{AHMTT} to
general Calder\'on--Zygmund operators is complicated. 

Hofmann \cite{Ho1} was able to extend to general Calder\'on--Zygmund operators but at the price of needing a stronger
set of assumptions: $\int_Q |b^1_Q|^s \le |Q|$, $\int_Q |b^2_Q|^s \le |Q|$, $\int_Q |Tb^1_Q|^{2} \le |Q|$ and $\int_Q |T^*b^2_Q|^{2} \le |Q|$ for some $s > 2$. Auscher and Yang \cite{AY}
established the theorem for standard Calder\'on--Zygmund operators in the case $1/p + 1/q \le 1$ (and thus in the case $p=q=2$).

We mention that there is also the question of considering the case $1/p + 1/q > 1$. While general exponents are not part of this paper, it has been an extremely active
area in the homogeneous world. Hofmann \cite{Ho2} has given a full solution in the case of square functions. In the Calder\'on--Zygmund world the work of Auscher and Routin \cite{AR}
continued to shed some light to the general case of exponents, however, not giving a definite answer and involving additional technical conditions. The (almost) full solution is given by Hyt\"onen and Nazarov \cite{HN}.

\section{Beginning of the proof}
Let $\lambda > 1$ be a fixed large constant. We begin by noting that by \cite{NTV1} there holds that
\begin{displaymath}
\|T\| \le C_1 + C_2\mathop{\sup_{Q_0 \subset \R^n}}_{Q_0 \textup{ cube}} \mathop{\sup_{f, g}}_{|f|, |g| \le 1_{Q_0}} \frac{|\langle Tf, g\rangle|}{\mu(\lambda Q_0)}.
\end{displaymath}
Let us fix a cube $Q_0 \subset \R^n$ and functions $f, g$ such that $|f|, |g| \le 1_{Q_0}$. It suffices to prove that
\begin{displaymath}
|\langle Tf, g\rangle| \le (C_3 + c\|T||)\mu(\lambda Q_0),
\end{displaymath}
where $c$ is so small that $C_2c \le 1/2$. 

For small notational convenience we assume that $c_{Q_0} = 0$ (that is, $Q_0$ is centred at the origin). 
Let $N \in \Z$ be defined by the requirement $2^{N-3} \le \ell(Q_0) < 2^{N-2}$.
Consider two independent random squares $Q^* = Q^*(w) = w + [-2^N,2^N)^n$ and $R^* = R^*(w') = w' + [-2^N, 2^N)^n$, where
$w, w' \in [-2^{N-1}, 2^{N-1})^n$. The cubes $Q^*$ and $R^*$ are taken to be the starting cubes of the independent grids $\mathcal{D}^T$ and $\mathcal{D}^{T^*}$ (only the cubes inside $Q^*$ and $R^*$ are included in these grids).
The probability measure is the normalized Lebesgue measure on the square $[-2^{N-1}, 2^{N-1})^n$. Furthermore, note that always spt$\,f$, spt$\,g \subset \alpha Q^* \cap \alpha R^*$ with some absolute constant $\alpha < 1$. 
There also holds that $Q^* \cup R^* \subset \lambda Q_0$ choosing $\lambda$ large enough.

\subsection{Martingale difference operators}
Let $M_{\mu}h(x) = \sup_{r > 0} \frac{1}{\mu(B(x,r))} \int_{B(x,r)} |h|\,d\mu$ be the centred maximal function. This is an $L^2(\mu)$ bounded operator.
Let us denote $\langle f \rangle_Q = \mu(Q)^{-1}\int_Q f\,d\mu$.
For a small convenience we may assume the normalisation $\langle b^T_Q \rangle_Q = 1 = \langle b^{T^*}_Q \rangle_Q$.
Let also $A$ denote the constant for which $\|b^T_Q\|_{L^2(\mu)}^2 + \|b^{T^*}_Q\|_{L^2(\mu)}^2 \le A\mu(Q)$ and $B$ the constant
for which $\|1_QTb^T_Q\|_{L^2(\mu)}^2 + \|1_QT^*b^{T^*}_Q\|_{L^2(\mu)}^2 \le B\mu(Q)$.

Let $\mathcal{F}_{Q^*}^1$ consist of the maximal cubes $Q \in \mathcal{D}^T$, $Q \subset Q^*$, for which at least one of the following three conditions holds:
\begin{enumerate}
\item $|\langle b^T_{Q^*} \rangle_Q| < 1/2$;
\item $\langle |M_{\mu}b^T_{Q*}|^2 \rangle_Q > 16A^2 \|M_{\mu}\|^2$; 
\item $\langle |Tb^T_{Q*}|^2 \rangle_Q > 16AB$.
\end{enumerate}
Next, one repeats the previous procedure by replacing $Q^*$ with a fixed $Q \in \mathcal{F}^1_{Q^*}$. The combined collection of stopping
cubes resulting from this is called $\mathcal{F}^2_{Q^*}$. This is continued and one sets $\mathcal{F}_{Q^*} = \bigcup_{j=0}^{\infty} \mathcal{F}^j_{Q^*}$.
Finally, for every $Q \in \mathcal{D}^T$, $Q \subset Q^*$, we let $Q^a \in \mathcal{F}_{Q^*}$ be the minimal
cube $S \in \mathcal{F}_{Q^*}$ for which $Q \subset S$.

Similarly, let $\mathcal{F}_{R^*}^1$ consist of the maximal cubes $R \in \mathcal{D}^{T^*}$, $R \subset R^*$, for which at least one of the following three conditions holds:
\begin{enumerate}
\item $|\langle b^{T^*}_{R^*} \rangle_R| < 1/2$;
\item $\langle |M_{\mu}b^{T^*}_{R*}|^2 \rangle_R > 16A^2 \|M_{\mu}\|^2$; 
\item $\langle |T^*b^{T^*}_{R*}|^2 \rangle_R > 16AB$.
\end{enumerate}
We define $\mathcal{F}_{R^*}$ and $R^a$ analogously as above.

The following results are proved essentially in \cite{LM:SF}.
\begin{lem}
For $F \in \mathcal{F}^j_{Q^*}$ there holds that
\begin{displaymath}
\mathop{\sum_{S \in \mathcal{F}^{j+1}_{Q^*}}}_{S \subset F} \mu(S) \le \Big(1 - \frac{1}{8A}\Big)\mu(F) =: \tau\mu(F), \qquad \tau < 1.
\end{displaymath}
\end{lem}
\begin{cor}
We have the following Carleson estimate:
\begin{displaymath}
\mathop{\sum_{F \in \mathcal{F}_{Q^*}}}_{F \subset Q} \mu(F) \lesssim \mu(Q), \qquad Q \in \mathcal{D}^T, \, Q \subset Q^*.
\end{displaymath}
\end{cor}

If $Q \in \mathcal{D}^T$, $Q \subset Q^*$, and $h \in L^2_{\textup{loc}}(\mu)$, we define the twisted martingale difference operators
\begin{equation*}
\Delta_Q h = \sum_{Q' \in \, \textrm{ch}(Q)} \Big[\frac{\langle h \rangle_{Q'}}{\langle b^T_{(Q')^a}\rangle_{Q'}}b^T_{(Q')^a} - \frac{\langle h \rangle_Q}{\langle b^T_{Q^a}\rangle_Q}b^T_{Q^a}\Big]1_{Q'}.
\end{equation*}
The operators $\Delta_R$, $R \in \mathcal{D}^{T^*}$, $R \subset R^*$, are analogously defined.

\subsection{General twisted martingale transform inequalities}
Most of the time we can make do with the following square function estimate.
\begin{lem}
Suppose $|h| \le 1$. Then there holds that
\begin{displaymath}
\sum_{Q \in \mathcal{D}^T} \|\Delta_Q h\|_{L^2(\mu)}^2 \lesssim \mu(Q^*).
\end{displaymath}
\end{lem}
\begin{proof}
The proof is like that of Proposition 2.5 of \cite{LM:SF} -- except easier because of the assumption $|h| \le 1$.
\end{proof}

But for a few key times (in connection with paraproducts) we absolutely depend on the variant presented in the next proposition.
It is much stronger in many sense but requires that we set $Q^a = F$ for a fixed $F \in \mathcal{F}_{Q^*}$. It is not a triviality to then sum over all the generations of stopping cubes $F$.
However, this is an issue that we do not care about since our proof of the local $Tb$ theorem incorporates a reduction to finitely many stopping generations. So we could use the bound that follows in every situation.
But just to stress that this stronger estimate and the reduction to finitely many generations is really needed only when dealing with the paraproduct, we mostly use the above bound.

We shall use the following notation. If $F \in \mathcal{F}_{Q^*}$, we let $j \in \N$ be such that $F \in \mathcal{F}_{Q^*}^j$ and define $\mathcal{H} = \mathcal{H}_F = \{H \in \mathcal{F}_{Q^*}^{j+1}:\, H \subset F\}$.
\begin{prop}\label{prop:twisted}
Suppose $F \in \mathcal{F}_{Q^*}$ and $h \in L^2(\mu)$. Suppose also that we have constants $\epsilon_Q$, $Q \in \mathcal{D}^T$, which satisfy $|\epsilon_Q| \le 1$. Then there holds that
\begin{displaymath}
\Big\|\mathop{\sum_{Q \in \mathcal{D}^T}}_{Q^a = F} \epsilon_Q \Delta_Q h\Big\|_{L^2(\mu)}^2 \lesssim \|h\|_{L^2(\mu)}^2.
\end{displaymath}
\end{prop}
\begin{proof}
Consider a cube $Q \in \mathcal{D}^T$ for which $Q^a = F$.
We define
\begin{displaymath}
D_Q h := \sum_{Q' \in \textup{ch}(Q) \setminus \mathcal{H}}  \Big[\frac{\langle h \rangle_{Q'}}{\langle b^T_{F}\rangle_{Q'}}- \frac{\langle h \rangle_Q}{\langle b^T_{F}\rangle_Q}\Big]1_{Q'}.
\end{displaymath}
Our aim is to reduce to these operators. However, for technical reasons certain maximal truncations of them will be needed.
Anyway, begin by noticing that
\begin{align*}
\Big\|\mathop{\sum_{Q \in \mathcal{D}^T}}_{Q^a = F} \epsilon_Q \Delta_Q h \cdot 1_{F \setminus \bigcup \mathcal{H}} \Big\|_{L^2(\mu)}^2 
&= \Big\|\mathop{\sum_{Q \in \mathcal{D}^T}}_{Q^a = F} \epsilon_Q D_Q h \cdot b^T_F1_{F \setminus \bigcup \mathcal{H}} \Big\|_{L^2(\mu)}^2 \\
&\lesssim \Big\|\mathop{\sum_{Q \in \mathcal{D}^T}}_{Q^a = F} \epsilon_Q D_Q h \Big\|_{L^2(\mu)}^2,
\end{align*}
since $b^T_F1_{F \setminus \bigcup \mathcal{H}} \in L^{\infty}(\mu)$.

We are reduced to controlling
\begin{displaymath}
\sum_{H \in \mathcal{H}} \Big\|\mathop{\sum_{Q \in \mathcal{D}^T}}_{H^{(1)} \subset Q \subset F} \epsilon_Q \Delta_Q h \cdot 1_H\Big\|_{L^2(\mu)}^2.
\end{displaymath}
We write
\begin{align*}
\mathop{\sum_{Q \in \mathcal{D}^T}}_{H^{(1)} \subset Q \subset F} \epsilon_Q \Delta_Q h \cdot 1_H = \epsilon_{H^{(1)}} \Big( \langle h \rangle_H b^T_H& - \frac{\langle h \rangle_{H^{(1)}}}{\langle b^T_F \rangle_{H^{(1)}}} b^T_F 1_H\Big) \\
&+ \mathop{\sum_{Q \in \mathcal{D}^T}}_{H^{(2)} \subset Q \subset F} \epsilon_Q D_Q h \cdot b^T_F1_H. 
\end{align*}
We have that
\begin{displaymath}
\sum_{H \in \mathcal{H}} |\epsilon_{H^{(1)}}|^2 | \langle h \rangle_{H}|^2 \|b^T_{H} \|_{L^2(\mu)}^2 \lesssim  \sum_{H \in \mathcal{H}} \int_H |h|^2\,d\mu \le \|h\|_{L^2(\mu)}^2.
\end{displaymath}
Next, notice that
\begin{align*}
\sum_{H \in \mathcal{H}} |\epsilon_{H^{(1)}}|^2 \frac{|\langle h \rangle_{H^{(1)}}|^2}{|\langle b^T_F \rangle_{H^{(1)}}|^2} \int_H |b^T_F|^2\,d\mu \lesssim \sum_{H \in \mathcal{H}} &|\langle h \rangle_{H^{(1)}}|^2  \int_H |b^T_F|^2\,d\mu \\
&= \sum_{Q \in \mathcal{D}^T} |\langle h \rangle_Q|^2 \alpha_Q \lesssim \|h\|_{L^2(\mu)}^2,
\end{align*}
since
\begin{displaymath}
\alpha_Q := \left\{ \begin{array}{ll} 
\sum_{Q' \in \textup{ch}(Q) \cap \mathcal{H}} \int_{Q'} |b^T_F|^2\,d\mu, & \textup{ if } Q^a = F, \\
0, & \textup{ otherwise},
\end{array} \right.
\end{displaymath}
is a Carleson sequence. Let us show this fact now.

To this end, fix a cube $R \in \mathcal{D}^T$. We have that
\begin{displaymath}
\sum_{Q \subset R} \alpha_Q = \mathop{\sum_{Q \subset R}}_{Q^a = F} \sum_{Q' \in \textup{ch}(Q) \cap \mathcal{H}} \int_{Q'} |b^T_F|^2\,d\mu \le \int_R |b^T_F|^2\,d\mu.
\end{displaymath}
We can assume that $R \subset F$ and that there is a $Q \subset R$ for which $Q^a = F$. But then $R^a = F$ and $\int_R |b^T_F|^2\,d\mu \lesssim \mu(R)$ proving the
Carleson property.

So we are to deal with
\begin{displaymath}
\sum_{H \in \mathcal{H}} \Big\| \mathop{\sum_{Q \in \mathcal{D}^T}}_{H^{(2)} \subset Q \subset F} \epsilon_Q D_Q h \cdot b^T_F1_H\Big\|_{L^2(\mu)}^2.
\end{displaymath}
But notice that
\begin{displaymath}
\sum_{Q:\, H^{(2)} \subset Q \subset F} \epsilon_Q D_Q h
\end{displaymath}
is constant on $H^{(1)}$. We can then estimate
\begin{align*}
\sum_{H \in \mathcal{H}} & \Big|\Big\langle \sum_{Q:\, H^{(2)} \subset Q \subset F} \epsilon_Q D_Q h\Big\rangle_{H^{(1)}} \Big|^2 \int_H |b^T_F|^2\,d\mu \\
&= \sum_{H \in \mathcal{H}}   \Big|\Big\langle \mathop{\mathop{\sum_{Q \in \mathcal{D}^T}}_{Q^a = F}}_{\ell(Q) > \ell(H^{(1)})} \epsilon_Q D_Q h\Big\rangle_{H^{(1)}} \Big|^2 \int_H |b^T_F|^2\,d\mu \\
&\le \sum_{H \in \mathcal{H}}  \Big\langle \sup_{\epsilon > 0} \Big| \mathop{\mathop{\sum_{Q \in \mathcal{D}^T}}_{Q^a = F}}_{\ell(Q) > \epsilon} \epsilon_Q D_Q h \Big|  \Big\rangle_{H^{(1)}}^2  \int_H |b^T_F|^2\,d\mu \\
&= \sum_{R \in \mathcal{D}^T}  \Big\langle \sup_{\epsilon > 0} \Big| \mathop{\mathop{\sum_{Q \in \mathcal{D}^T}}_{Q^a = F}}_{\ell(Q) > \epsilon} \epsilon_Q D_Q h \Big|  \Big\rangle_{R}^2 \alpha_R
 \lesssim \Big\| \sup_{\epsilon > 0} \Big| \mathop{\mathop{\sum_{Q \in \mathcal{D}^T}}_{Q^a = F}}_{\ell(Q) > \epsilon} \epsilon_Q D_Q h \Big|\, \Big\|_{L^2(\mu)}^2.
\end{align*}
The conclusion is that the proposition follows from the $L^2(\mu)$ bound of these maximal truncations. But this bound follows from Proposition \ref{prop:lpmax}.
\end{proof}

To control the maximal truncations we need some heavier tools. For the reader's convenience we formulate and prove the needed principles exactly. The first is a reduction
to a testing condition (a dyadic adaptation of Sawyer's idea \cite{Sa} and also essentially contained in \cite{seven}). The second is a non-homogeneous John--Nirenberg principle (essentially found in \cite{LPR} and \cite{HCTV}).

Suppose that for every $Q \in \mathcal{D}^T$ we are given an operator $A_Q$ satisfying:
\begin{itemize}
\item $A_Q h = \sum_{Q' \in \textup{ch}(Q)} c_{Q'}(h) 1_{Q'}$ for some constants $c_{Q'}(h)$;
\item $|A_Qh| \le M_{\mu}^{\mathcal{D}^T}h$, where $M_{\mu}^{\mathcal{D}^T}h(x) := \sup_{R \in \mathcal{D}^T} 1_R(x) \mu(R)^{-1}\int_R |h|\,d\mu$.
\end{itemize}
We set
\begin{align*}
A_{\epsilon}h &:= \mathop{\sum_{Q \in \mathcal{D}^T}}_{\ell(Q) > \epsilon} A_Q h, \qquad \epsilon > 0, \\
A_{\#}h &:= \sup_{\epsilon > 0} |A_{\epsilon} h|.
\end{align*}
For $P \in \mathcal{D}^T$, let us define
\begin{align*}
A_{\epsilon}^{P}h &=   \mathop{\mathop{\sum_{Q \in \mathcal{D}^T}}_{Q \subset P}}_{\ell(Q) > \epsilon} A_Q h, \qquad \epsilon > 0, \\
A_{\#}^P h &:= \sup_{\epsilon > 0} |A_{\epsilon}^P h|.
\end{align*}

The following lemma states that a certain testing condition for maximal truncations implies an estimate $L^p(\mu) \to L^{p, \infty}(\mu)$ for the maximal truncations.
\begin{lem}\label{lem:test}
Let $p \in (1,\infty)$. We assume that for every $Q \in \mathcal{D}^T$ and $h \in L^p(\mu)$ there holds that
\begin{equation}\label{eq:maxtest}
\int_Q A_{\#}^Q h\,d\mu \lesssim \|h1_Q\|_{L^p(\mu)} \mu(Q)^{1/p'}.
\end{equation}
Then for every $h \in L^p(\mu)$ there holds that $\|A_{\#}h\|_{L^{p, \infty}(\mu)} \lesssim \|h\|_{L^p(\mu)}$.
\end{lem}
\begin{proof}
We assume qualitatively that $A_Q \ne 0$ for only finitely many $Q \in \mathcal{D}^T$. This gives us the a priori information $\|A_{\#}h\|_{L^{p, \infty}(\mu)} < \infty$ for $h \in L^p(\mu)$.

Let $\lambda > 0$ and set $\Omega_{\lambda} := \{A_{\#}h > \lambda\}$. Let $M_{\lambda}$ consist of the maximal $Q \in \mathcal{D}^T$ for which $Q \subset \Omega_{\lambda}$. It is clear that
$\Omega_{\lambda} = \bigcup_{Q \in M_{\lambda}} Q$, since for every $x \in \Omega_{\lambda}$ there is a $Q \in \mathcal{D}^T$ such that $x \in Q \subset \Omega_{\lambda}$.

Let us set
\begin{displaymath}
H_{\lambda}(Q) := Q \cap \{A_{\#}h > 4\lambda, \, M_{\mu}^{\mathcal{D}^T}h \le \lambda\}, \qquad Q \in M_{\lambda}.
\end{displaymath}
We have that
\begin{displaymath}
\Omega_{4\lambda} \subset \bigcup_{Q \in M_{\lambda}} H_{\lambda}(Q) \cup \{M_{\mu}^{\mathcal{D}^T}h > \lambda\}
\end{displaymath}
so that we have for any $\eta > 0$ that
\begin{align*}
\lambda^p \mu(\Omega_{4\lambda}) &\le \lambda^p \sum_{Q \in M_{\lambda}} \mu(H_{\lambda}(Q)) + \|M_{\mu}^{\mathcal{D}^T}h\|_{L^{p, \infty}(\mu)}^p \\
&\le \mathop{\lambda^p \sum_{Q \in M_{\lambda}}}_{\mu(H_{\lambda}(Q)) > \eta\mu(Q)} \mu(H_{\lambda}(Q)) + \eta \lambda^p \mu(\Omega_{\lambda}) + C\|h\|_{L^p(\mu)}^p.
\end{align*}

Let $Q \in M_{\lambda}$ and $x \in Q$. Let $z \in Q^{(1)}$ be such that $A_{\#}h(z) \le \lambda$. Let $\epsilon > 0$ be arbitrary. If $\ell(Q) > \epsilon$ we have
the following identity
\begin{align*}
A_{\epsilon}h(x) = A^Q_{\epsilon}h(x) +  \mathop{\sum_{R \in \mathcal{D}^T}}_{\ell(Q) < \ell(R) \le 4\ell(Q)} A_R h(x) +
\mathop{\sum_{R \in \mathcal{D}^T}}_{4\ell(Q) < \ell(R)} A_R h(z).
\end{align*}
We get the bound that
\begin{displaymath}
|A_{\epsilon}h(x)| \le A_{\#}^Q h(x) + 2M_{\mu}^{\mathcal{D}^T}h(x) + A_{\#}h(z).
\end{displaymath}
It is clear that this bound holds for every $\epsilon > 0$. Therefore, we have that
\begin{displaymath}
4\lambda < A_{\#}h(x) \le A_{\#}^Q h(x) + 3\lambda, \qquad Q \in M_{\lambda}, \, x \in H_{\lambda}(Q).
\end{displaymath}
This yields that if $Q \in M_{\lambda}$ and  $\mu(H_{\lambda}(Q)) > \eta\mu(Q)$, then there holds that
\begin{displaymath}
\int_Q A_{\#}^Q h \,d\mu \ge \int_{H_{\lambda}(Q)} A_{\#}^Q h \,d\mu \ge \lambda \mu(H_{\lambda}(Q)) \ge \eta\lambda \mu(Q).
\end{displaymath}
From this we can conclude that
\begin{displaymath}
\Big(\frac{1}{\mu(Q)} \int_Q A_{\#}^Q h \,d\mu\Big)^p \ge \eta^p\lambda^p, \qquad Q \in M_{\lambda}, \, \mu(H_{\lambda}(Q)) > \eta\mu(Q).
\end{displaymath}

We can now see using the assumed testing condition that
\begin{align*}
\mathop{\lambda^p \sum_{Q \in M_{\lambda}}}_{\mu(H_{\lambda}(Q)) > \eta\mu(Q)} \mu(H_{\lambda}(Q)) \le C\eta^{-p} \sum_{Q \in M_{\lambda}} \|h1_Q\|_{L^p(\mu)}^p \le C\eta^{-p} \|h\|_{L^p(\mu)}^p.
\end{align*}
We have shown that
\begin{displaymath}
4^{-p}(4\lambda)^p\mu(\Omega_{4\lambda}) \le C\eta^{-p} \|h\|_{L^p(\mu)}^p + \eta \lambda^p \mu(\Omega_{\lambda}) + C\|h\|_{L^p(\mu)}^p.
\end{displaymath}
This yields that
\begin{displaymath}
4^{-p} \|A_{\#}h\|_{L^{p,\infty}(\mu)}^p \le C\eta^{-p} \|h\|_{L^p(\mu)}^p + \eta  \|A_{\#}h\|_{L^{p,\infty}(\mu)}^p + C\|h\|_{L^p(\mu)}^p.
\end{displaymath}
Taking $\eta = 4^{-p}/2$ and using the fact that $\|A_{\#}h\|_{L^{p, \infty}(\mu)} < \infty$ we get the claim.
\end{proof}


The following two lemmata capture our usage of the non-homogeneous John--Nirenberg principle.
\begin{lem}\label{lem:p1}
Suppose that for every $P \in \mathcal{D}^T$ there holds that
\begin{displaymath}
\|A_{\#}^P 1\|_{L^1(\mu)} \lesssim \mu(P).
\end{displaymath}
Then for every $p \in (1,\infty)$ and for every $P \in \mathcal{D}^T$ there holds that
\begin{displaymath}
\|A_{\#}^P 1\|_{L^p(\mu)} \lesssim \mu(P)^{1/p}.
\end{displaymath}
\end{lem}
\begin{proof}
Follows from the next lemma by taking $\varphi_Q = A_Q 1 / C$ for a large enough constant $C > 1$.
\end{proof}
\begin{lem}
Assume that for every $Q \in \mathcal{D}^T$ we are given a function $\varphi_Q$ such that
\begin{itemize}
\item $\varphi_Q = \sum_{Q' \in \textup{ch}(Q)} c_{Q'} 1_{Q'}$ for some constants $c_{Q'}$;
\item $\|\varphi_Q\|_{L^{\infty}(\mu)} \le 1$.
\end{itemize}
For every $P \in \mathcal{D}^T$ we set
\begin{displaymath}
\Phi_P := \sup_{\epsilon > 0} \Big| \mathop{\mathop{\sum_{Q \in \mathcal{D}^T}}_{Q \subset P}}_{\ell(Q) > \epsilon} \varphi_Q \Big|.
\end{displaymath}
Suppose that for every $P \in \mathcal{D}^T$ there holds that
\begin{displaymath}
\mu(\{x \in P: \, \Phi_P(x) > 1\}) \le \mu(P)/2.
\end{displaymath}
Then for every $P \in \mathcal{D}^T$ and $t > 1$ there holds that
\begin{displaymath}
\mu(\{x \in P: \Phi_P(x) > t\}) \le 2^{-(t-1)/2} \mu(P).
\end{displaymath}
\end{lem}
\begin{proof}
Fix $P_0 \in \mathcal{D}^T$. Let $\mathcal{R}_1$ consist of the maximal $R \in \mathcal{D}^T$ such that $R \subset P_0$ and
\begin{displaymath}
\Big| \sum_{Q: \, R \subsetneq Q \subset P_0} \varphi_Q(x)\Big| > 1, \qquad x \in R.
\end{displaymath}
The left-hand side is constant on $R$ so this makes sense. Define $S_1 := \bigcup_{R \in \mathcal{R}_1} R$. We have that:
\begin{itemize}
\item $\mu(S_1) \le \mu(\{x \in P_0: \, \Phi_{P_0}(x) > 1\}) \le \mu(P_0)/2$;
\item $\Phi_{P_0}1_{P_0 \setminus S_1} \le 1$;
\item
For $R \in \mathcal{R}_1$ and $x \in R$ we have that
\begin{displaymath}
\Big| \sum_{Q: \, R \subsetneq Q \subset P_0} \varphi_Q(x)\Big| \le 2.
\end{displaymath}
\end{itemize}
For $R_0 \in \mathcal{R}_1$ we let $\mathcal{R}_2^{R_0}$ consist of the maximal $R \in \mathcal{D}^T$ such that $R \subset R_0$ and
\begin{displaymath}
\Big| \sum_{Q: \, R \subsetneq Q \subset R_0} \varphi_Q(x)\Big| > 1, \qquad x \in R.
\end{displaymath}
Let $\mathcal{R}_2 := \bigcup_{R_0 \in \mathcal{R}_1} \mathcal{R}_2^{R_0}$ and $S_2 := \bigcup_{R \in \mathcal{R}_2} R$. We have that:
\begin{itemize}
\item $\mu(S_2) \le \mu(P_0)/4$;
\item $\Phi_{P_0}1_{P_0 \setminus S_2} \le 3$;
\item For $R \in \mathcal{R}_2$ and $x \in R$ we have that
\begin{displaymath}
\Big| \sum_{Q: \, R \subsetneq Q \subset P_0} \varphi_Q(x)\Big| \le 4.
\end{displaymath}
\end{itemize}

Continue like this. We establish collections $\mathcal{R}_j$ and sets $S_j = \bigcup_{R \in \mathcal{R}_j} R$ such that there holds
$\mu(S_j) \le 2^{-j}\mu(P_0)$ and $\Phi_{P_0}1_{P_0 \setminus S_j} \le 2j-1$. Let $t > 1$ and choose $j_t \in \N$ such that
$2j_t - 1 \le t < 2j_t + 1$. We have that
\begin{displaymath}
\mu(\{x \in P_0: \Phi_{P_0}(x) > t\}) \le \mu(S_{j_t}) \le 2^{-j_t}\mu(P_0) \le 2^{-(t-1)/2}\mu(P_0).
\end{displaymath}
\end{proof}

An important tool for us is the following standard maximal truncation estimate for martingale differences.
\begin{lem}
Suppose we have constants $\epsilon_Q$, $Q \in \mathcal{D}^T$, which satisfy $|\epsilon_Q| \le 1$. Let $p \in (1, \infty)$. We have for every $h \in L^p(\mu)$ the estimate
\begin{displaymath}
\Big\| \sup_{\epsilon > 0 } \Big| \mathop{\sum_{Q \in \mathcal{D}^T}}_{\ell(Q) > \epsilon} \epsilon_Q \sum_{Q' \in \textup{ch}(Q)} [\langle h \rangle_{Q'} - \langle h \rangle_{Q}]1_{Q'} \Big|\, \Big\|_{L^p(\mu)}^p \lesssim \|h\|_{L^p(\mu)}^p.
\end{displaymath}
\end{lem}
We need a version of this where we have removed the stopping children.
\begin{cor}\label{cor:maxmar}
Suppose $F \in \mathcal{F}_{Q^*}$. Suppose also that we have constants $\epsilon_Q$, $Q \in \mathcal{D}^T$, which satisfy $|\epsilon_Q| \le 1$.
Let $p \in (1, \infty)$. We have for every $h \in L^p(\mu)$ the estimate
\begin{displaymath}
\Big\| \sup_{\epsilon > 0 } \Big| \mathop{\mathop{\sum_{Q \in \mathcal{D}^T}}_{Q^a = F}}_{\ell(Q) > \epsilon} 
\epsilon_Q \sum_{Q' \in \textup{ch}(Q) \setminus \mathcal{H}} [\langle h \rangle_{Q'} - \langle h \rangle_{Q}]1_{Q'} \Big|\, \Big\|_{L^p(\mu)}^p \lesssim \|h\|_{L^p(\mu)}^p.
\end{displaymath}
\end{cor}
\begin{proof}
Notice that
\begin{align*}
\sup_{\epsilon > 0} \Big|  \mathop{\mathop{\sum_{Q \in \mathcal{D}^T}}_{Q^a = F}}_{\ell(Q) > \epsilon} 
\epsilon_Q \sum_{Q' \in \textup{ch}(Q) \cap \mathcal{H}} [\langle h \rangle_{Q'} - \langle h \rangle_{Q}]1_{Q'}(x) \Big|
&\le \sum_{H \in \mathcal{H}} |\langle h \rangle_H - \langle h \rangle_{H^{(1)}}|1_H(x) \\
&\le 2 \sum_{H \in \mathcal{H}} 1_H(x)M^{\mathcal{D}^T}_{\mu}h(x).
\end{align*}
But then we have that
\begin{align*}
\Big\| \sup_{\epsilon > 0 } \Big| \mathop{\mathop{\sum_{Q \in \mathcal{D}^T}}_{Q^a = F}}_{\ell(Q) > \epsilon} 
\epsilon_Q \sum_{Q' \in \textup{ch}(Q) \cap \mathcal{H}} [\langle h \rangle_{Q'} - \langle h \rangle_{Q}]1_{Q'} \Big|\, \Big\|_{L^p(\mu)}^p &\lesssim  \sum_{H \in \mathcal{H}} \|1_H M^{\mathcal{D}^T}_{\mu}h\|_{L^p(\mu)}^p \\
&\le \|M^{\mathcal{D}^T}_{\mu}h\|_{L^p(\mu)}^p \lesssim  \|h\|_{L^p(\mu)}^p.
\end{align*}
Combining this with the previous lemma we have the result.
\end{proof}

The proof of Proposition \ref{prop:lpmax} will be based on a reduction to the testing condition (Lemma \ref{lem:test}). However, to verify the testing condition we still require the following lemma. It is in the proof
of this final lemma that the John--Nirenberg type reductions from above are used.
\begin{lem}\label{lem:D1}
Suppose $F \in \mathcal{F}_{Q^*}$. Suppose also that we have constants $\epsilon_Q$, $Q \in \mathcal{D}^T$, which satisfy $|\epsilon_Q| \le 1$.
For every $p \in [1,\infty)$ and $P \in \mathcal{D}^T$ there holds that
\begin{displaymath}
\Big\| \sup_{\epsilon > 0} \Big| \mathop{\mathop{\sum_{Q \in \mathcal{D}^T}}_{Q^a = F, \, Q \subset P}}_{\ell(Q) > \epsilon} \epsilon_Q D_Q 1 \Big|\, \Big\|_{L^p(\mu)}^p \lesssim \mu(P).
\end{displaymath}
\end{lem}
\begin{proof}
By Lemma \ref{lem:p1} it suffices to prove that for every $P \in \mathcal{D}^T$ there holds that
\begin{displaymath}
\int_P \sup_{\epsilon > 0} \Big| \mathop{\mathop{\sum_{Q \in \mathcal{D}^T}}_{Q^a = F, \, Q \subset P}}_{\ell(Q) > \epsilon} \epsilon_Q D_Q 1 \Big| \,d\mu \lesssim \mu(P).
\end{displaymath}
Let us write
\begin{displaymath}
\frac {1} {\langle b^T_F \rangle _{Q'}  } - \frac {1} {\langle b^T_F \rangle_Q}  =
	\frac {{\langle b^T_F \rangle_Q}  -  {\langle b^T_F \rangle _{Q'}  }} {\langle b^T_F \rangle _{Q} ^2     } +
	\frac { {[\langle b^T_F \rangle_Q}  -  {\langle b^T_F \rangle _{Q'}  }]^2 } {{\langle b^T_F \rangle_Q}^2\langle b^T_F \rangle _{Q'}     }.
\end{displaymath}

Define $\tilde \epsilon_Q := \epsilon_Q / \langle b^T_F \rangle _{Q} ^2$, $Q^a  = F$. Note that $|\tilde \epsilon_Q| \lesssim 1$, and then that
\begin{align*}
&\int_P \sup_{\epsilon > 0} \Big| \mathop{\mathop{\sum_{Q \in \mathcal{D}^T}}_{Q^a = F, \, Q \subset P}}_{\ell(Q) > \epsilon} \tilde \epsilon_Q \sum_{Q' \in \textup{ch}(Q) \setminus \mathcal{H}}
[\langle b^T_F \rangle _{Q'}   - {\langle b^T_F \rangle_Q}]1_{Q'} \Big| \,d\mu \\
& \le \mu(P)^{1/2} \Big( \int_P \Big[ \sup_{\epsilon > 0} \Big| \mathop{\mathop{\sum_{Q \in \mathcal{D}^T}}_{Q^a = F, \, Q \subset P}}_{\ell(Q) > \epsilon} \tilde \epsilon_Q \sum_{Q' \in \textup{ch}(Q) \setminus \mathcal{H}}
[\langle b^T_F \rangle _{Q'}   - {\langle b^T_F \rangle_Q}]1_{Q'} \Big|\Big]^2 \,d\mu\Big)^{1/2} \\
&\lesssim \mu(P)^{1/2}\|1_Pb^T_F\|_{L^2(\mu)} \lesssim \mu(P).
\end{align*}
Here we first appealed to the $L ^2 $ bound for maximal truncations of a martingale difference (Corollary \ref{cor:maxmar}). For the last inequality we have the following explanation. It is trivial if $F \cap P = \emptyset$ or $F \subset P$. Otherwise,
we may assume that there is a $Q$ for which $Q^a = F$ and $Q \subset P \subset F$. But then $P^a = F$.

Next, notice that
\begin{align*}
&\int_P \sup_{\epsilon > 0} \Big| \mathop{\mathop{\sum_{Q \in \mathcal{D}^T}}_{Q^a = F, \, Q \subset P}}_{\ell(Q) > \epsilon} \epsilon_Q \sum_{Q' \in \textup{ch}(Q) \setminus \mathcal{H}}
\frac { {[\langle b^T_F \rangle_Q}  -  {\langle b^T_F \rangle _{Q'}  }]^2 } {{\langle b^T_F \rangle_Q}^2\langle b^T_F \rangle _{Q'}} 1_{Q'} \Big| \,d\mu \\
&\lesssim \mathop{\sum_{Q \in \mathcal{D}^T}}_{Q^a = F, \, Q \subset P} \int \Big[ \sum_{Q' \in \textup{ch}(Q)} (\langle b^T_F \rangle_{Q'}  -  \langle b^T_F \rangle _{Q})1_{Q'} \Big]^2\,d\mu \lesssim \|1_Pb^T_F\|_{L^2(\mu)}^2 \lesssim \mu(P).
\end{align*}

\end{proof}

\begin{prop}\label{prop:lpmax}
Suppose $F \in \mathcal{F}_{Q^*}$. Suppose also that we have constants $\epsilon_Q$, $Q \in \mathcal{D}^T$, which satisfy $|\epsilon_Q| \le 1$. Then for every $p \in (1,\infty)$ and $h \in L^p(\mu)$ there holds that
\begin{displaymath}
\Big\| \sup_{\epsilon > 0} \Big| \mathop{\mathop{\sum_{Q \in \mathcal{D}^T}}_{Q^a = F}}_{\ell(Q) > \epsilon} \epsilon_Q D_Q h \Big|\, \Big\|_{L^p(\mu)}^p \lesssim \|h\|_{L^p(\mu)}^p.
\end{displaymath}
\end{prop}
\begin{proof}
Fix $1 < p < \infty$, $h \in L^p(\mu)$ and $P \in \mathcal{D}^T$.
By Lemma \ref{lem:test} we need to prove that there holds that
\begin{displaymath}
\int_P \sup_{\epsilon > 0} \Big| \mathop{\mathop{\sum_{Q \in \mathcal{D}^T}}_{Q^a = F, \, Q \subset P}}_{\ell(Q) > \epsilon} \epsilon_Q D_Q h \Big| \,d\mu \lesssim \|h1_P\|_{L^p(\mu)} \mu(P)^{1/p'}.
\end{displaymath}
Indeed, then we have the weak type bound for every $p$ and we can interpolate the sublinear operator to establish the strong type bounds.

We now write
\begin{align} \notag 
	\frac {\langle h \rangle _{Q'}} {\langle b^T_F \rangle _{Q'}  } 
	-
	\frac {\langle h \rangle _{Q}} {\langle b^T_F \rangle_Q}  
	&=   
	 \Bigl\{
	\frac {\langle h \rangle _{Q'}} {\langle b^T_F \rangle _{Q}  } 
	-	\frac {\langle h \rangle _{Q}} {\langle b^T_F \rangle_Q}  
	\Bigr\}+ 
	\Bigl\{
		\frac {\langle h \rangle _{Q'}} {\langle b^T_F \rangle _{Q'}  } - 
			\frac {\langle h \rangle _{Q'}} {\langle b^T_F \rangle _{Q}  } 
	\Bigr\} 
	\\ \label{e:mt1} 
	&= \frac 1 {\langle b^T_F \rangle _{Q}  } 
	\bigl\{ {\langle h \rangle _{Q'}}  -  {\langle h \rangle _{Q}} 
	\Bigr\} 
	\\ \label{e:mt2}
	& \qquad +	
	\bigl\{ {\langle h \rangle _{Q'}}  -  {\langle h \rangle _{Q}} \bigr\} 
	\Bigl\{
	\frac 1 { \langle b^T_F \rangle_{Q'}} - 
	\frac 1 {\langle  b^T_F\rangle_{Q} } 
	\Bigr\}
	\\ \label{e:mt3}& \qquad + 
	\langle h\rangle_Q 
	\Bigl\{
	\frac 1 { \langle b^T_F \rangle_{Q'}} - 
	\frac 1 {\langle  b^T_F\rangle_{Q} } 
	\Bigr\}
\end{align}
This leaves us with three terms to control. 

Define $\tilde \epsilon_Q := \epsilon_Q / \langle b^T_F \rangle _{Q}$, $Q^a  = F$. Note that $|\tilde \epsilon_Q| \lesssim 1$.
The control of \eqref{e:mt1} goes as follows:
\begin{align*}
& \int_P \sup_{\epsilon > 0} \Big| \mathop{\mathop{\sum_{Q \in \mathcal{D}^T}}_{Q^a = F, \, Q \subset P}}_{\ell(Q) > \epsilon}
 \tilde \epsilon_Q \sum_{Q' \textup{ch}(Q) \setminus \mathcal{H}} [\langle h \rangle _{Q'} -  \langle h \rangle _Q]1_{Q'}\Big|\,d\mu \\
& \le \Big\| \sup_{\epsilon > 0} \Big| \mathop{\mathop{\sum_{Q \in \mathcal{D}^T}}_{Q^a = F, \, Q \subset P}}_{\ell(Q) > \epsilon}
 \tilde \epsilon_Q \sum_{Q' \textup{ch}(Q) \setminus \mathcal{H}} [\langle h \rangle _{Q'} -  \langle h \rangle _Q]1_{Q'}\Big|\, \Big\|_{L^p(\mu)} \mu(P)^{1/p'} \\
 &\lesssim \|h1_P\|_{L^p(\mu)} \mu(P)^{1/p'}.
\end{align*}
Here we used Corollary \ref{cor:maxmar}.

We will then control \eqref{e:mt2}. Let us define
\begin{displaymath}
\Delta^{c}_Q h = \sum_{Q' \in \textup{ch}(Q)} [\langle h \rangle _{Q'} -  \langle h \rangle _Q]1_{Q'},
\end{displaymath}
where c stands for classical. Notice that
\begin{displaymath}
\Delta^{c}_Q h \cdot D_Q 1 = \sum_{Q' \textup{ch}(Q) \setminus \mathcal{H}} \bigl\{ {\langle h \rangle _{Q'}}  -  {\langle h \rangle _{Q}} \bigr\} 
	\Bigl\{
	\frac 1 { \langle b^T_F \rangle_{Q'}} - 
	\frac 1 {\langle  b^T_F\rangle_{Q} } 
	\Bigr\}1_{Q'}.
\end{displaymath}
The small point we want to make is that the other martingale can in fact be taken classical, since it is multiplied with $D_Q$ which is supported on the children of $Q$ which are not in $\mathcal{H}$.
Now we have that
\begin{displaymath}
\Big| \mathop{\mathop{\sum_{Q \in \mathcal{D}^T}}_{Q^a = F, \, Q \subset P}}_{\ell(Q) > \epsilon} \epsilon_Q \Delta^{c}_Qh \cdot  D_Q 1\Big| \le \Big( \sum_{Q \in \mathcal{D}^T} |\Delta^{c}_Q(h1_P)|^2 \Big)^{1/2}
\Big( \mathop{\sum_{Q \in \mathcal{D}^T}}_{Q^a = F, \, Q \subset P} |D_Q1|^2 \Big)^{1/2}.
\end{displaymath}
It is enough to note that
\begin{displaymath}
\Big\|  \Big( \sum_{Q \in \mathcal{D}^T} |\Delta^{c}_Q(h1_P)|^2 \Big)^{1/2} \Big\|_{L^p(\mu)} \Big\| \Big( \mathop{\sum_{Q \in \mathcal{D}^T}}_{Q^a = F, \, Q \subset P} |D_Q1|^2 \Big)^{1/2} \Big\|_{L^{p'}(\mu)} \lesssim \|h1_P\|_{L^p(\mu)} \mu(P)^{1/p'}.
\end{displaymath}
To control the last term we used Lemma \ref{lem:D1}. Indeed, this form follows from it by averaging over independent random signs $\pm 1$.

We are left to control the term with \eqref{e:mt3}. To control the averages $\langle h \rangle_Q$ in front, we will perform a standard stopping time.
Let $\mathcal{S}_0 = \{P\}$. Let $\mathcal{S}_1$ consist of the maximal $R \in \mathcal{D}^T$, $R \subset P$, for which $\langle |h| \rangle_R > 4\langle |h|\rangle_P$.
Continuing this in the standard way we get the full stopping tree $\mathcal{S} = \bigcup_{j=0}^{\infty} \mathcal{S}_j$. For $Q \in \mathcal{D}^T$, $Q \subset P$, we define
$Q^s$ to be the minimal $S \in \mathcal{S}$ for which $Q \subset S$. We have that $\langle |h| \rangle_Q \le 4\langle |h|\rangle_{Q^s}$.

If $Q^s = S$ we let $\epsilon_Q(S) := \epsilon_Q \frac{ \langle h \rangle_Q}{ \langle |h| \rangle_S}$. Notice that $|\epsilon_Q(S)| \le 4|\epsilon_Q| \lesssim 1$.
We then estimate using the $p=1$ case of Lemma \ref{lem:D1}:
\begin{align*}
& \int_P \sup_{\epsilon > 0} \Big| \sum_{S \in \mathcal{S}} \mathop{\mathop{\sum_{Q \in \mathcal{D}^T}}_{Q^a = F, \, Q^s = S}}_{\ell(Q) > \epsilon} \epsilon_Q \langle h \rangle_Q D_Q 1 \Big|\,d\mu \\
&\le  \sum_{S \in \mathcal{S}} \int_S  \sup_{\epsilon > 0} \Big| \mathop{\mathop{\sum_{Q \in \mathcal{D}^T}}_{Q^a = F, \, Q^s = S}}_{\ell(Q) > \epsilon} \epsilon_Q \langle h \rangle_Q D_Q 1 \Big|\,d\mu \\
&=  \sum_{S \in \mathcal{S}} \langle |h| \rangle_S \int_S \sup_{\epsilon > 0}  \Big| \mathop{\mathop{\sum_{Q \in \mathcal{D}^T}}_{Q^a = F, \, Q^s = S}}_{\ell(Q) > \epsilon} \epsilon_Q(S) D_Q 1 \Big|\,d\mu \\
&\lesssim \sum_{S \in \mathcal{S}} \langle |h| \rangle_S \mu(S) \le \Big( \sum_{S \in \mathcal{S}} \langle |h| \rangle_S^p \mu(S) \Big)^{1/p} \Big( \sum_{S \in \mathcal{S}} \mu(S) \Big)^{1/p'} \lesssim \|h1_P\|_{L^p(\mu)} \mu(P)^{1/p'}.
\end{align*}
This completes the proof of the proposition.
\end{proof}

\begin{rem}
We only need the following conclusion of Proposition \ref{prop:twisted}. If $|h| \le 1$ and $S \subset \R^n$ is an arbitrary set, then there holds that
\begin{displaymath}
\Big\|\mathop{\sum_{Q \in \mathcal{D}^T:\, Q \subset S}}_{Q^a = F} \epsilon_Q \Delta_Q h\Big\|_{L^2(\mu)}^2 = \Big\|\mathop{\sum_{Q \in \mathcal{D}^T:\, Q \subset S}}_{Q^a = F} \epsilon_Q \Delta_Q (1_{S \cap F} h)\Big\|_{L^2(\mu)}^2
\lesssim \mu(S \cap F).
\end{displaymath}
\end{rem}

\subsection{Further reductions}
We now expand (see Proposition 2.8 of \cite{LM:SF})
\begin{displaymath}
f = \mathop{\sum_{Q \in \mathcal{D}^T}} \Delta_Q f + \langle f \rangle_{Q^*} b^T_{Q^*}
\end{displaymath}
and
\begin{displaymath}
g =  \mathop{\sum_{R \in \mathcal{D}^{T^*}}} \Delta_R g + \langle g \rangle_{R^*} b^{T^*}_{R^*}.
\end{displaymath}
If $Q \in \mathcal{D}^T$ is such that $Q^a \in \mathcal{F}^j_{Q^*}$, we define $\beta(Q) := j$. Let $\beta > 0$ be a large parameter (we shall fix it momentarily).
We have
\begin{align*}
|\langle Tf, g\rangle| \le \Big| \mathop{\sum_{Q \in \mathcal{D}^T}}_{\beta(Q) < \beta} \langle T(\Delta_Q f), g\rangle\Big| + \Big| \mathop{\sum_{Q \in \mathcal{D}^T}}_{\beta(Q) \ge \beta} \langle T(\Delta_Q f), g\rangle\Big|
+ |\langle Tb^T_{Q^*}, g\rangle|.
\end{align*}
Notice that
\begin{align*}
\Big| \mathop{\sum_{Q \in \mathcal{D}^T}}_{\beta(Q) \ge \beta} \langle T(\Delta_Q f), g\rangle\Big| &\le \sum_{j=\beta}^{\infty}
 \|T\| \Big\| \sum_{F \in \mathcal{F}^j_{Q^*}} \mathop{\sum_{Q \in \mathcal{D}^T}}_{Q^a = F} \Delta_Q f\Big\|_{L^2(\mu)} \|g\|_{L^2(\mu)} \\
 &\le \|T\| \mu(Q_0)^{1/2} \sum_{j=\beta}^{\infty} \Big(  \sum_{F \in \mathcal{F}^j_{Q^*}} \Big\|\mathop{\sum_{Q \in \mathcal{D}^T}}_{Q^a = F} \Delta_Q f\Big\|_{L^2(\mu)}^2 \Big)^{1/2} \\
 &\lesssim \|T\| \mu(Q_0)^{1/2} \sum_{j=\beta}^{\infty} \Big(  \sum_{F \in \mathcal{F}^j_{Q^*}} \mu(F) \Big)^{1/2} \\
 &\le \|T\| \mu(Q_0)^{1/2} \sum_{j=\beta}^{\infty} \tau^{j/2} \mu(Q^*)^{1/2} \lesssim \tau^{\beta/2} \|T\| \mu(\lambda Q_0)
\end{align*}
and $|\langle Tb^T_{Q^*}, g\rangle| \le \|1_{Q^*} Tb^T_{Q^*}\|_{L^2(\mu)} \|g\|_{L^2(\mu)} \lesssim \mu(\lambda Q_0)$.

Next, we have
\begin{align*}
\mathop{\sum_{Q \in \mathcal{D}^T}}_{\beta(Q) < \beta} \langle T(\Delta_Q f), g\rangle &= \mathop{\sum_{Q \in \mathcal{D}^T}}_{\beta(Q) < \beta} \mathop{\sum_{R \in \mathcal{D}^{T^*}}}_{\beta(R) < \beta} \langle T(\Delta_Q f), \Delta_R g\rangle
+ \Big\langle T\Big( \mathop{\sum_{Q \in \mathcal{D}^T}}_{\beta(Q) < \beta} \Delta_Q f\Big),  \mathop{\sum_{R \in \mathcal{D}^{T^*}}}_{\beta(R) \ge \beta} \Delta_R g\Big\rangle \\
&+ \langle g \rangle_{R^*} \langle f,   1_{R^*}T^*b^{T^*}_{R^*} \rangle -  \langle g \rangle_{R^*} \Big\langle T\Big( \mathop{\sum_{Q \in \mathcal{D}^T}}_{\beta(Q) \ge \beta} \Delta_Q f\Big), b^{T^*}_{R^*} \Big\rangle \\
& -\langle f\rangle_{Q^*} \langle g \rangle_{R^*} \langle  Tb^T_{Q^*},   b^{T^*}_{R^*} \rangle.
\end{align*}
Again, there holds that
\begin{displaymath}
\Big| \Big\langle T\Big( \mathop{\sum_{Q \in \mathcal{D}^T}}_{\beta(Q) < \beta} \Delta_Q f\Big),  \mathop{\sum_{R \in \mathcal{D}^{T^*}}}_{\beta(R) \ge \beta} \Delta_R g\Big\rangle \Big| \lesssim \beta \tau^{\beta/2} \|T\| \mu(\lambda Q_0)
\end{displaymath}
and
\begin{displaymath}
\Big|  \langle g \rangle_{R^*} \Big\langle T\Big( \mathop{\sum_{Q \in \mathcal{D}^T}}_{\beta(Q) \ge \beta} \Delta_Q f\Big), b^{T^*}_{R^*} \Big\rangle \Big| \lesssim \tau^{\beta/2} \|T\| \mu(\lambda Q_0).
\end{displaymath}
Also, we have that $|\langle g \rangle_{R^*} \langle f,   1_{R^*}T^*b^{T^*}_{R^*} \rangle| \lesssim \mu(\lambda Q_0)$.

The pairing $\langle  Tb^T_{Q^*},   b^{T^*}_{R^*} \rangle$ is trickier. Let $u > 0$. We estimate
\begin{align*}
|\langle  Tb^T_{Q^*},   b^{T^*}_{R^*} \rangle| &\le  \|1_{Q^*} Tb^T_{Q^*}\|_{L^2(\mu)} \| b^{T^*}_{R^*}\|_{L^2(\mu)} + \|T\| \|b^T_{Q^*}\|_{L^2(\mu)} \|1_{(1+u)Q^* \setminus Q^*} b^{T^*}_{R^*}\|_{L^2(\mu)} \\
&+ \int_{R^* \setminus (1+u)Q^*} \int_{Q^*} \frac{C(u)}{\ell(Q^*)^{m/2} \ell(R^*)^{m/2}} |b^T_{Q^*}(y)| |b_{R^*}^{T^*}(x)|\,d\mu(y)\,d\mu(x) \\
&\le C(u) \mu(\lambda Q_0) +  \|T\| \mu(\lambda Q_0)^{1/2} \|1_{(1+u)Q^* \setminus Q^*} b^{T^*}_{R^*}\|_{L^2(\mu)}.
\end{align*}
Notice that with a fixed $w'$ we have that
\begin{align*}
E_{w} \|1_{(1+u)Q^* \setminus Q^*} b^{T^*}_{R^*}\|_{L^2(\mu)} \le \Big( \int_{R^*} P_w(x \in (1+u)Q^* \setminus Q^*) |&b^{T^*}_{R^*}(x)|\,d\mu(x)\Big)^{1/2} \\
&\le c(u)\mu(\lambda Q_0)^{1/2},
\end{align*}
where $c(u) \to 0 $ when $u \to 0$.

The conclusion of this subsection is that
\begin{align*}
|\langle Tf, g\rangle| \le E_{w,w'} \Big|&\mathop{\sum_{Q \in \mathcal{D}^T}}_{\beta(Q) < \beta} \mathop{\sum_{R \in \mathcal{D}^{T^*}}}_{\beta(R) < \beta} \langle T(\Delta_Q f), \Delta_R g\rangle\Big| \\
&+ C(u)\mu(\lambda Q_0) + c(\beta) \|T\| \mu(\lambda Q_0) + c(u) \|T\| \mu(\lambda Q_0),
\end{align*}
where $c(\beta) \to 0 $ when $\beta \to 0$ and $c(u) \to 0 $ when $u \to 0$. We now fix $\beta$ and $u$ to be so small that $(c(\beta) + c(u))C_2 \le 1/4$.
In the sequel some estimates will depend on the fixed parameter $\beta$ but this is no longer a concern (and the dependance will not be tracked).
We may now focus on proving that
\begin{displaymath}
E_{w,w'} \Big|\mathop{\sum_{Q \in \mathcal{D}^T}}_{\beta(Q) < \beta} \mathop{\sum_{R \in \mathcal{D}^{T^*}}}_{\beta(R) < \beta} \langle T(\Delta_Q f), \Delta_R g\rangle\Big| \le (C_4 + \tilde c \|T\|)\mu(\lambda Q_0),
\end{displaymath}
where $\tilde c$ is so small that $C_2\tilde c \le 1/4$.

\subsection{Splitting of the summation}
We set $\gamma = \alpha/(2m+2\alpha)$, where $\alpha > 0$ appears in the kernel estimates and $m$ appears in $\mu(B(x,t)) \lesssim t^m$. We also let $r > 0$ be a large constant that we shall fix
later. We will also focus on the part of the summation where $\ell(Q) < \ell(R)$. We will simply split this sum in to three standard pieces:
\begin{itemize}
\item $Q$: $\ell(Q) < \ell(R)$ and $d(Q, R) > \ell(Q)^{\gamma}\ell(R)^{1-\gamma}$;
\item $Q$: $\ell(Q) \le 2^{-r}\ell(R)$ and $d(Q, R) \le \ell(Q)^{\gamma}\ell(R)^{1-\gamma}$;
\item $Q$: $2^{-r}\ell(R) < \ell(Q) < \ell(R)$ and $d(Q, R) \le \ell(Q)^{\gamma}\ell(R)^{1-\gamma}$.
\end{itemize}
We call the first sum the separated sum, the second sum the nested sum and the last sum the diagonal sum. Here the term nested is the most cryptic, but
will be justified using probability by introducing good cubes to the second sum in a specific way (like in \cite{HM}).

In the next section we will prove that
\begin{displaymath}
\mathop{\sum_{R \in \mathcal{D}^{T^*}}}_{\beta(R) < \beta} \mathop{\mathop{\sum_{Q \in \mathcal{D}^T:\, \beta(Q) < \beta}}_{\ell(Q) < \ell(R)}}_{d(Q,R) > \ell(Q)^{\gamma}\ell(R)^{1-\gamma}}
|\langle T(\Delta_Q f), \Delta_R g\rangle| \lesssim \mu(\lambda Q_0).
\end{displaymath}
In a section after that we will show that
\begin{displaymath}
E_{w,w'} \Big| \mathop{\sum_{R \in \mathcal{D}^{T^*}}}_{\beta(R) < \beta} \mathop{\mathop{\sum_{Q \in \mathcal{D}^T:\, \beta(Q) < \beta}}_{\ell(Q) \le 2^{-r}\ell(R)}}_{d(Q, R) \le \ell(Q)^{\gamma}\ell(R)^{1-\gamma}}
\langle T(\Delta_Q f), \Delta_R g\rangle\Big| \le C\mu(\lambda Q_0) + c(r)\|T\|\mu(\lambda Q_0),
\end{displaymath}
where $c(r) \to 0$ as $r \to 0$. We may then fix the parameter $r$ at this point of the argument to be so small that $C_2c(r) \le 1/16$.
The estimates of the last sum may depend on $r$, but this is no longer a concern (and the dependance will not be tracked). In the last section we will prove that
\begin{displaymath}
E_{w,w'} \mathop{\sum_{R \in \mathcal{D}^{T^*}}}_{\beta(R) < \beta} \mathop{\mathop{\sum_{Q \in \mathcal{D}^T:\, \beta(Q) < \beta}}_{2^{-r}\ell(R) < \ell(Q) < \ell(R)}}_{d(Q, R) \le \ell(Q)^{\gamma}\ell(R)^{1-\gamma}}
| \langle T(\Delta_Q f), \Delta_R g\rangle| \le C\mu(\lambda Q_0) + \hat c\|T\|\mu(\lambda Q_0),
\end{displaymath}
where $\hat c C_2 \le 1/16$. Combining with the symmetric argument for the case $\ell(Q) \ge \ell(R)$ this proves our main theorem.

\section{The separated sum}
If $\ell(Q) < \ell(R)$ and $d(Q,R) > \ell(Q)^{\gamma}\ell(R)^{1-\gamma}$, then
\begin{displaymath}
|\langle T(\Delta_Q f), \Delta_R g\rangle| \lesssim A_{QR} \|\Delta_Q f\|_{L^2(\mu)}  \|\Delta_R g\|_{L^2(\mu)},
\end{displaymath}
where
\begin{align*}
A_{QR} &:= \frac{\ell(Q)^{\alpha/2}\ell(R)^{\alpha/2}}{D(Q,R)^{m+\alpha}} \mu(Q)^{1/2}\mu(R)^{1/2}; \\
D(Q,R) &:= \ell(Q) + \ell(R) + d(Q,R).
\end{align*}
Moreover, by \cite{NTV1} this $\ell^2$ estimate holds 
\begin{displaymath}
\sum_{Q, R} A_{QR} x_Q y_R \lesssim \Big( \sum_Q x_Q^2\Big)^{1/2} \Big( \sum_R y_R^2\Big)^{1/2}.
\end{displaymath}
Therefore, we have that
\begin{align*}
\mathop{\sum_{R \in \mathcal{D}^{T^*}}}_{\beta(R) < \beta} &\mathop{\mathop{\sum_{Q \in \mathcal{D}^T:\, \beta(Q) < \beta}}_{\ell(Q) < \ell(R)}}_{d(Q,R) > \ell(Q)^{\gamma}\ell(R)^{1-\gamma}}
|\langle T(\Delta_Q f), \Delta_R g\rangle| \\
& \lesssim \Big( \sum_Q \|\Delta_Q f\|_{L^2(\mu)}^2\Big)^{1/2} \Big( \sum_R \|\Delta_R g\|_{L^2(\mu)}^2\Big)^{1/2}\\
& \lesssim \mu(Q^*)^{1/2}\mu(R^*)^{1/2} \le \mu(\lambda Q_0).
\end{align*}

\section{The nested sum}
Define $\mathcal{D}^T_{\textrm{bad}, \, A}$ to be the collection of those cubes $Q \in \mathcal{D}^T$ which are bad with respect to \emph{some} $\mathcal{D}^{T^*}$-cube of
side length $A$ or larger. We define that this demands that there should exist a cube $S \in \mathcal{D}^{T^*}$ for which $\ell(S) \ge A$ and $d(Q, \textrm{sk}\,S) \le \ell(Q)^{\gamma}\ell(S)^{1-\gamma}$, where
$\textup{sk}\, S := \bigcup_{S' \in \textup{ch}(S)} \partial S'$. Let $\mathcal{D}^T_{\textrm{good},\,A}$ be the collection of those $Q \in \mathcal{D}^T$ which are good with respect to \emph{all} $\mathcal{D}^{T^*}$-cubes of side length
$A$ and larger. This means that for every $S \in \mathcal{D}^{T^*}$ for which $\ell(S) \ge A$ there holds that $d(Q, \textrm{sk}\,S) > \ell(Q)^{\gamma}\ell(S)^{1-\gamma}$.
\begin{rem}
Notice carefully the usage of the words \emph{some} and \emph{all} above.
\end{rem}

Let us write
\begin{displaymath}
\mathop{\sum_{R \in \mathcal{D}^{T^*}}}_{\beta(R) < \beta} \mathop{\mathop{\sum_{Q \in \mathcal{D}^T:\, \beta(Q) < \beta}}_{\ell(Q) \le 2^{-r}\ell(R)}}_{d(Q, R) \le \ell(Q)^{\gamma}\ell(R)^{1-\gamma}}
\langle T(\Delta_Q f), \Delta_R g\rangle = S_{\textup{good}} + S_{\textup{bad}},
\end{displaymath}
where
\begin{displaymath}
S_{\textup{good}} = \mathop{\sum_{R \in \mathcal{D}^{T^*}}}_{\beta(R) < \beta} 
\mathop{\mathop{\sum_{Q \in \mathcal{D}^T_{\textrm{good},\,\ell(R)}:\, \beta(Q) < \beta}}_{\ell(Q) \le 2^{-r}\ell(R)}}_{d(Q, R) \le \ell(Q)^{\gamma}\ell(R)^{1-\gamma}}
\langle T(\Delta_Q f), \Delta_R g\rangle
\end{displaymath}
and
\begin{displaymath}
S_{\textup{bad}} = \mathop{\sum_{R \in \mathcal{D}^{T^*}}}_{\beta(R) < \beta} 
\mathop{\mathop{\sum_{Q \in \mathcal{D}^T_{\textrm{bad},\,\ell(R)}:\, \beta(Q) < \beta}}_{\ell(Q) \le 2^{-r}\ell(R)}}_{d(Q, R) \le \ell(Q)^{\gamma}\ell(R)^{1-\gamma}}
\langle T(\Delta_Q f), \Delta_R g\rangle.
\end{displaymath}

\subsection{The bad part is small}
Notice that for a given $R \in \mathcal{D}^{T^*}$ there holds that
\begin{displaymath}
\Big\| \mathop{\mathop{\sum_{Q \in \mathcal{D}^T_{\textrm{bad},\,\ell(R)}:\, \beta(Q) < \beta}}_{\ell(Q) \le 2^{-r}\ell(R)}}_{d(Q, R) \le \ell(Q)^{\gamma}\ell(R)^{1-\gamma}} \Delta_Q f\Big\|_{L^2(\mu)}
\le \sum_{k=r}^{\infty} \Big(\mathop{\mathop{\sum_{Q \in \mathcal{D}^T_{\textrm{bad},\,2^k\ell(Q)}}}_{\ell(Q) = 2^{-k}\ell(R)}}_{d(Q, R) \le \ell(Q)^{\gamma}\ell(R)^{1-\gamma}} \|\Delta_Q f\|_{L^2(\mu)}^2 \Big)^{1/2}
\end{displaymath}
so that
\begin{align*}
|S_{\textup{bad}}| &\le \|T\|\sum_{k=r}^{\infty} \sum_{R \in \mathcal{D}^{T^*}} 
\Big(\mathop{\mathop{\sum_{Q \in \mathcal{D}^T_{\textrm{bad},\,2^k\ell(Q)}}}_{\ell(Q) = 2^{-k}\ell(R)}}_{d(Q, R) \le \ell(Q)^{\gamma}\ell(R)^{1-\gamma}} \|\Delta_Q f\|_{L^2(\mu)}^2 \Big)^{1/2} \|\Delta_R g\|_{L^2(\mu)} \\
&\le C\|T\| \sum_{k=r}^{\infty} \Big( \sum_{Q \in \mathcal{D}^T_{\textrm{bad},\,2^k\ell(Q)}} \|\Delta_Q f\|_{L^2(\mu)}^2\Big)^{1/2} \Big(  \sum_{R \in \mathcal{D}^{T^*}}  \|\Delta_R g\|_{L^2(\mu)}^2 \Big)^{1/2} \\
&\le C\|T\| \mu(\lambda Q_0)^{1/2} \sum_{k=r}^{\infty} \Big( \sum_{Q \in \mathcal{D}^T_{\textrm{bad},\,2^k\ell(Q)}} \|\Delta_Q f\|_{L^2(\mu)}^2\Big)^{1/2},
\end{align*}
where we used that
\begin{displaymath}
\#\{ R \in \mathcal{D}^{T^*}:\, \ell(R) = 2^k\ell(Q) \textup{ and } d(Q, R) \le \ell(Q)^{\gamma}\ell(R)^{1-\gamma}\} \lesssim 1, \qquad k \ge 0.
\end{displaymath}
We conclude that
\begin{align*}
E_{w, w'} |S_{\textup{bad}}| \le C\|T\| \mu(\lambda Q_0)^{1/2}E_w &\sum_{k=r}^{\infty} \Big( \sum_{Q \in \mathcal{D}^T} \mathbb{P}_{w'}(Q \in \mathcal{D}^T_{\textrm{bad},\,2^k\ell(Q)}) \|\Delta_Q f\|_{L^2(\mu)}^2\Big)^{1/2} \\
&\le C\|T\| \mu(\lambda Q_0) \sum_{k=r}^{\infty} 2^{-\gamma k/2} = c(r)\|T\| \mu(\lambda Q_0),
\end{align*}
where $c(r) \to 0$ as $r \to \infty$. We now fix $r$ so that $c(r)C_2 \le 1/16$.

\subsection{The good part}
Note that if $Q \in \mathcal{D}^T$ is good with respect to $R \in \mathcal{D}^{T^*}$ and $d(Q,R) \le \ell(Q)^{\gamma}\ell(R)^{1-\gamma}$, then there is a child $R_Q \in \textup{ch}(R)$ so that
$Q \subset R_Q$ and $d(Q, R_Q^c) > \ell(Q)^{\gamma}\ell(R)^{1-\gamma}$. So the good part is indeed a nested sum in the sense that $Q$ is deep inside $R$.

Before having to split the argument into a case study, we prove two lemmata.
\begin{lem}\label{lem:es1}
If $R \in \mathcal{D}^{T^*}$, $Q \in \mathcal{D}^T_{\textrm{good},\,\ell(R)}$ and $Q \subset R$, then
there holds that
\begin{align*}
|\langle T(\Delta_Q f), 1_{R_Q^c}b^{T^*}_{R^a}\rangle| \lesssim  \Big(\frac{\ell(Q)}{\ell(R)}\Big)^{\alpha/2} \mu(Q)^{1/2} \|\Delta_Q f\|_{L^2(\mu)}.
\end{align*}
\end{lem}
\begin{proof}
We will first show that
\begin{displaymath}
 \int_{R^a \setminus R_Q} \frac{|b^{T^*}_{R^a}(x)|}{|x-c_Q|^{m+\alpha}} \,d\mu(x) \lesssim \ell(Q)^{-\alpha/2}\ell(R)^{-\alpha/2}.
\end{displaymath}
Let $M$ be such that $R_Q^{(M+1)} = R^a$. Notice that since $Q$ is good with respect to all $S \in \mathcal{D}^{T^*}$ for which $\ell(S) \ge \ell(R)$, there holds that
\begin{displaymath}
d(Q, \partial R_Q^{(j)})^{m+\alpha} \gtrsim \ell(Q)^{\alpha/2} \ell(R_Q^{(j)})^{m+\alpha/2} \gtrsim 2^{\alpha j/2} \ell(Q)^{\alpha/2} \ell(R)^{\alpha/2} \mu(R_Q^{(j+1)}).
\end{displaymath}
Here we used that $\gamma(m+\alpha) = \alpha/2$.

We may now estimate
\begin{align*}
 \int_{R^a \setminus R_Q} \frac{|b^{T^*}_{R^a}(x)|}{|x-c_Q|^{m+\alpha}} \,d\mu(x) &= \sum_{j=0}^M \int_{R_Q^{(j+1)} \setminus R_Q^{(j)}} \frac{|b^{T^*}_{R^a}(x)|}{|x-c_Q|^{m+\alpha}} \,d\mu(x) \\
 &\le \sum_{j=0}^M \frac{1}{d(Q, \partial R_Q^{(j)})^{m+\alpha}} \int_{R_Q^{(j+1)}} |b^{T^*}_{R^a}(x)| \,d\mu(x) \\
 &\lesssim \sum_{j=0}^M \frac{\mu(R_Q^{(j+1)})}{2^{\alpha j/2} \ell(Q)^{\alpha/2} \ell(R)^{\alpha/2} \mu(R_Q^{(j+1)})} \lesssim \ell(Q)^{-\alpha/2}\ell(R)^{-\alpha/2}.
\end{align*}

To end the proof it remains to use the H\"older estimate of $K$ to get that
\begin{align*}
|\langle T(\Delta_Q f), 1_{R_Q^c}b^{T^*}_{R^a}\rangle| \lesssim \|\Delta_Q f\|_{L^1(\mu)} \cdot \ell(Q)^{\alpha}& \int_{R^a \setminus R_Q} \frac{|b^{T^*}_{R^a}(x)|}{|x-c_Q|^{m+\alpha}} \,d\mu(x) \\
&\lesssim \Big(\frac{\ell(Q)}{\ell(R)}\Big)^{\alpha/2} \mu(Q)^{1/2} \|\Delta_Q f\|_{L^2(\mu)}.
\end{align*}
\end{proof}

\begin{lem}\label{lem:es2}
Let $R \in \mathcal{D}^{T^*}$, $Q \in \mathcal{D}^T$ be good with respect to $R$ and $Q \subset R$.
Then there holds that
\begin{displaymath}
|\langle T(\Delta_Q f), 1_{R_Q^c}\Delta_R g\rangle| \lesssim B_{QR} \|\Delta_Q f\|_{L^2(\mu)} \|\Delta_R g\|_{L^2(\mu)},
\end{displaymath}
where
\begin{displaymath}
B_{QR} :=  \Big(\frac{\ell(Q)}{\ell(R)}\Big)^{\alpha/2} \Big(\frac{\mu(Q)}{\mu(R_Q)}\Big)^{1/2}.
\end{displaymath}
\end{lem}
\begin{proof}
Let $S \in \textup{ch}(R)$, $S \ne R_Q$. Then $d(Q,S) \ge d(Q, \partial R_Q) \ge \ell(Q)^{\gamma}\ell(R)^{1-\gamma}$. Using this it is easy to see (like in the separated sum) that there holds that
\begin{displaymath}
|\langle T(\Delta_Q f), 1_S\Delta_R g\rangle| \lesssim \Big(\frac{\ell(Q)}{\ell(R)}\Big)^{\alpha/2} \mu(Q)^{1/2} \frac{\mu(S)^{1/2}}{\ell(R)^m} \|\Delta_Q f\|_{L^2(\mu)} \|\Delta_R g\|_{L^2(\mu)}.
\end{displaymath} 
The claim follows from this since
\begin{displaymath}
\frac{\mu(S)^{1/2}}{\ell(R)^m} \lesssim \frac{1}{\ell(R)^{m/2}} \lesssim \mu(R_Q)^{-1/2}. 
\end{displaymath}
\end{proof}

Notice that
\begin{align*}
\sum_{R \in \mathcal{D}^{T^*}} &
 \mathop{\mathop{\sum_{Q \in \mathcal{D}^T_{\textrm{good},\,\ell(R)}}}_{\ell(Q) \le 2^{-r}\ell(R)}}_{Q \subset R}
|\langle T(\Delta_Q f), 1_{R_Q^c}\Delta_R g\rangle| \\ 
&\lesssim \sum_{R \in \mathcal{D}^{T^*}} \mathop{\mathop{\sum_{Q \in \mathcal{D}^T}}_{\ell(Q) \le 2^{-r}\ell(R)}}_{Q \subset R_Q \in \textup{ch}(R)}
B_{QR} \|\Delta_Q f\|_{L^2(\mu)} \|\Delta_R g\|_{L^2(\mu)} \\
&\lesssim \Big( \sum_Q \|\Delta_Q f\|_{L^2(\mu)}^2\Big)^{1/2} \Big( \sum_R \|\Delta_R g\|_{L^2(\mu)}^2\Big)^{1/2}  \lesssim \mu(Q^*)^{1/2}\mu(R^*)^{1/2} \le \mu(\lambda Q_0).
\end{align*}
Here we used Lemma \ref{lem:es2} and the fact that by \cite{NTV1} we have the $\ell^2$ estimate
\begin{displaymath}
 \sum_{R \in \mathcal{D}^{T^*}} \mathop{\mathop{\sum_{Q \in \mathcal{D}^T}}_{\ell(Q) \le 2^{-r}\ell(R)}}_{Q \subset R_Q \in \textup{ch}(R)} B_{QR}x_Q y_R \lesssim \Big( \sum_Q x_Q^2\Big)^{1/2} \Big( \sum_R y_R^2\Big)^{1/2}.
\end{displaymath}
Therefore, we need to only consider
\begin{displaymath}
S_{\textup{good}}' := \mathop{\sum_{R \in \mathcal{D}^{T^*}}}_{\beta(R) < \beta} 
\mathop{\mathop{\sum_{Q \in \mathcal{D}^T_{\textrm{good},\,\ell(R)}:\, \beta(Q) < \beta}}_{\ell(Q) \le 2^{-r}\ell(R)}}_{d(Q, R) \le \ell(Q)^{\gamma}\ell(R)^{1-\gamma}}
\langle T(\Delta_Q f), 1_{R_Q}\Delta_R g\rangle.
\end{displaymath}

\subsubsection*{The case $R_Q^a = R^a$}
Define
\begin{displaymath}
C_{R_Q} := \frac{\langle g \rangle_{R_Q}}{ \langle b^{T^*}_{R_Q^a} \rangle_{R_Q}} -  \frac{\langle g \rangle_{R}}{ \langle b^{T^*}_{R^a} \rangle_{R}}.
\end{displaymath}
Writing $1_{R_Q} = 1 - 1_{R_Q^c}$ we see that
\begin{displaymath}
1_{R_Q}\Delta_R g = C_{R_Q}1_{R_Q} b^{T^*}_{R^a} = C_{R_Q}b^{T^*}_{R^a} - C_{R_Q}1_{R_Q^c}b^{T^*}_{R^a}.
\end{displaymath}
The first part will become part of the paraproduct and we do not touch it further in this subsection.

Notice that now
\begin{displaymath}
|C_{R_Q}| \mu(R_Q) \lesssim \Big| \int_{R_Q} C_{R_Q} b^{T^*}_{R^a}\,d\mu\Big| = \Big| \int_{R_Q} \Delta_R g\,d\mu\Big| \le \mu(R_Q)^{1/2}\|\Delta_R g\|_{L^2(\mu)}.
\end{displaymath}
Therefore, we have using Lemma \ref{lem:es1} that
\begin{align*}
\sum_{R \in \mathcal{D}^{T^*}}&
\mathop{\mathop{\sum_{Q \in \mathcal{D}^T_{\textrm{good},\,\ell(R)}: \,Q \subset R}}_{\ell(Q) \le 2^{-r}\ell(R)}}_{R_Q^a = R^a}
|C_{R_Q}| |\langle T(\Delta_Q f), 1_{R_Q^c}b^{T^*}_{R^a}\rangle| \\
&\lesssim \sum_{R \in \mathcal{D}^{T^*}} \mathop{\mathop{\sum_{Q \in \mathcal{D}^T}}_{\ell(Q) \le 2^{-r}\ell(R)}}_{Q \subset R_Q \in \textup{ch}(R)}
B_{QR} \|\Delta_Q f\|_{L^2(\mu)} \|\Delta_R g\|_{L^2(\mu)} \lesssim \mu(\lambda Q_0).
\end{align*}

\subsubsection*{The case $R_Q^a = R_Q$}
We now write
\begin{align*}
1_{R_Q} \Delta_R g = \Big( \frac{\langle g \rangle_{R_Q}}{ \langle b^{T^*}_{R_Q^a} \rangle_{R_Q}}b^{T^*}_{R_Q^a} -  \frac{\langle g \rangle_{R}}{ \langle b^{T^*}_{R^a} \rangle_{R}}b^{T^*}_{R^a} \Big) +
\frac{\langle g \rangle_{R}}{ \langle b^{T^*}_{R^a} \rangle_{R}}b^{T^*}_{R^a} 1_{R_Q^c}.
\end{align*}
The first part is exactly the same thing that we did not touch previously, and we will not do so here either. It will become part of the paraproduct.

But let us notice that Lemma \ref{lem:es1} again gives that
\begin{align*}
\sum_{R \in \mathcal{D}^{T^*}}&
\mathop{\mathop{\sum_{Q \in \mathcal{D}^T_{\textrm{good},\,\ell(R)}: \,Q \subset R}}_{\ell(Q) \le 2^{-r}\ell(R)}}_{R_Q^a = R_Q} \Big|\langle T(\Delta_Q f), \frac{\langle g \rangle_{R}}{ \langle b^{T^*}_{R^a} \rangle_{R}}b^{T^*}_{R^a} 1_{R_Q^c}\rangle\Big| \\
&\lesssim \sum_{R \in \mathcal{D}^{T^*}} 
 \mathop{\mathop{\sum_{Q \in \mathcal{D}^T}}_{\ell(Q) \le 2^{-r}\ell(R)}}_{Q \subset R_Q \in \textup{ch}(R)} B_{QR} \|\Delta_Q f\|_{L^2(\mu)} \Big( \mathop{\sum_{R' \in \textup{ch}(R)}}_{(R')^a = R'} \mu(R')\Big)^{1/2} |\langle g \rangle_R| \\
 &\lesssim \Big( \sum_Q \|\Delta_Q f\|_{L^2(\mu)}^2\Big)^{1/2} \Big( \sum_R |\langle g\rangle_R|^2  \mathop{\sum_{R' \in \textup{ch}(R)}}_{(R')^a = R'} \mu(R') \Big)^{1/2}\\ &\lesssim \mu(Q^*)\|g\|_{L^2(\mu)} \le \mu(\lambda Q_0).
\end{align*}

\subsubsection*{The paraproduct}
For $Q \in \bigcup_{k \ge r} \mathcal{D}^T_{\textup{good}, \, 2^k\ell(Q)}$ let $\alpha(Q)$ be the smallest $k$ such that $Q \in \mathcal{D}^T_{\textup{good}, \, 2^k\ell(Q)}$.
We are reduced to bounding
\begin{align*}
S_{\textup{par}} := &\mathop{\sum_{R \in \mathcal{D}^{T^*}}}_{\beta(R) < \beta} 
\mathop{\mathop{\sum_{Q \in \mathcal{D}^T_{\textrm{good},\,\ell(R)}:\, \beta(Q) < \beta}}_{\ell(Q) \le 2^{-r}\ell(R)}}_{d(Q, R) \le \ell(Q)^{\gamma}\ell(R)^{1-\gamma}}
\Big\langle T(\Delta_Q f),  \frac{\langle g \rangle_{R_Q}}{ \langle b^{T^*}_{R_Q^a} \rangle_{R_Q}}b^{T^*}_{R_Q^a} -  \frac{\langle g \rangle_{R}}{ \langle b^{T^*}_{R^a} \rangle_{R}}b^{T^*}_{R^a}  \Big\rangle \\
&= \mathop{\sum_{Q \in \bigcup_{k \ge r} \mathcal{D}^T_{\textup{good}, \, 2^k\ell(Q)}:\, \beta(Q) < \beta}}_{Q \subset R^*} \mathop{\mathop{\sum_{R \in \mathcal{D}^{T^*}: \, \beta(R) < \beta}}_{\ell(R) \ge 2^{\alpha(Q)}\ell(Q)}}_{Q \subset R}
\Big\langle T(\Delta_Q f),  \frac{\langle g \rangle_{R_Q}}{ \langle b^{T^*}_{R_Q^a} \rangle_{R_Q}}b^{T^*}_{R_Q^a} -  \frac{\langle g \rangle_{R}}{ \langle b^{T^*}_{R^a} \rangle_{R}}b^{T^*}_{R^a}  \Big\rangle.
\end{align*}
If $Q \in \bigcup_{k \ge r} \mathcal{D}^T_{\textup{good}, \, 2^k\ell(Q)}$ we let $H(Q) \in \mathcal{D}^{T^*}$ be the smallest cube satisfying
$\beta(H(Q)) < \beta$, $\ell(H(Q)) \ge 2^{\alpha(Q)}\ell(Q)$ and $Q \subset H(Q)$. Let $J(Q) = H(Q)_Q$. We have that
\begin{align*}
S_{\textup{par}} &= \mathop{\sum_{Q \in \bigcup_{k \ge r} \mathcal{D}^T_{\textup{good}, \, 2^k\ell(Q)}:\, \beta(Q) < \beta}}_{Q \subset R^*} \Big\langle T(\Delta_Q f), \mathop{\sum_{R \in \mathcal{D}^{T^*}}}_{H(Q) \subset R \subset R^*}
 \frac{\langle g \rangle_{R_Q}}{ \langle b^{T^*}_{R_Q^a} \rangle_{R_Q}}b^{T^*}_{R_Q^a} -  \frac{\langle g \rangle_{R}}{ \langle b^{T^*}_{R^a} \rangle_{R}}b^{T^*}_{R^a}  \Big\rangle \\
 &= \mathop{\sum_{Q \in \bigcup_{k \ge r} \mathcal{D}^T_{\textup{good}, \, 2^k\ell(Q)}:\, \beta(Q) < \beta}}_{Q \subset R^*} \Big\langle 
 T(\Delta_Q f),
 \frac{\langle g \rangle_{J(Q)}}{ \langle b^{T^*}_{J(Q)^a} \rangle_{J(Q)}}b^{T^*}_{J(Q)^a} -  \frac{\langle g \rangle_{R^*}}{ \langle b^{T^*}_{R^{*a}} \rangle_{R^*}}b^{T^*}_{R^{*a}}  \Big\rangle. 
\end{align*}

We may consider the following general situation.
We are given a collection $\mathcal{G} \subset \mathcal{D}^T$ so that to every cube $Q \in \mathcal{G}$ there holds $\beta(Q) < \beta$, and there is associated a unique cube $S(Q) \in \mathcal{D}^{T^*}$ satisfying $Q \subset S(Q)$.
Our object is to bound
\begin{displaymath}
P(f,g) := \Big| \sum_{Q \in \mathcal{G}} \Big\langle T(\Delta_Q f), \frac{\langle g \rangle_{S(Q)}}{\langle b^{T^*}_{S(Q)^a} \rangle_{S(Q)}} b^{T^*}_{S(Q)^a} \Big\rangle\Big|.
\end{displaymath}

To this end, we first define for $F \in \mathcal{F}_{R^*}$ and $Q \in \mathcal{D}^T$ that
\begin{displaymath}
\epsilon_Q(F) = \left\{ \begin{array}{ll} 
0 & \textup{ if } Q \not \in \mathcal{G}, \\
0 & \textup{ if } Q \in \mathcal{G} \textup{ and } S(Q)^a \ne F, \\
\frac{\langle g \rangle_{S(Q)}}{\langle b^{T^*}_{F} \rangle_{S(Q)}} & \textup{ if } Q \in \mathcal{G} \textup{ and } S(Q)^a = F.
\end{array} \right.
\end{displaymath}
Notice that $|\epsilon_Q(F)| \lesssim |\langle g \rangle_{S(Q)}| \le 1$. We have that
\begin{align*}
\Big\| \sum_{Q \in \mathcal{D}^T}  \epsilon_Q(F) \Delta_Q f\Big\|_{L^2(\mu)} &\le \sum_{j=0}^{\beta-1} \Big( \sum_{K \in \mathcal{F}^j_{Q^*}} \Big\| \mathop{\sum_{Q \in \mathcal{D}^T}}_{Q^a = K} \epsilon_Q(F) \Delta_Q f\ \Big\|_{L^2(\mu)}^2\Big)^{1/2}
\\ &\lesssim \sum_{j=0}^{\beta-1} \Big( \sum_{K \in \mathcal{F}^j_{Q^*}} \mu(F \cap K) \Big)^{1/2} \le \beta \mu(F)^{1/2}.
\end{align*}
But as $\beta$ is already fixed we do not need to mind about this dependence.
Using this we now have that
\begin{align*}
P(f,g) &= \Big| \mathop{\sum_{F \in \mathcal{D}^{T^*}}}_{F^a = F} \Big\langle \mathop{\sum_{Q \in \mathcal{G}}}_{S(Q)^a = F}  \frac{\langle g \rangle_{S(Q)}}{\langle b^{T^*}_{F} \rangle_{S(Q)}} \Delta_Q f, 1_F T^*  b^{T^*}_{F} \Big\rangle \Big| \\
&\le \mathop{\sum_{F \in \mathcal{D}^{T^*}}}_{F^a = F} \Big\| \sum_{Q \in \mathcal{D}^T}  \epsilon_Q(F) \Delta_Q f\Big\|_{L^2(\mu)} \| 1_F T^*  b^{T^*}_{F} \|_{L^2(\mu)} \\
&\lesssim  \mathop{\sum_{F \in \mathcal{D}^{T^*}}}_{F^a = F}  \mu(F) \lesssim \mu(R^*) \le \mu(\lambda Q_0).
\end{align*}

In particular, we have shown that
\begin{displaymath}
|S_{\textup{par}}| \lesssim \mu(\lambda Q_0).
\end{displaymath}
This completes our proof of the fact that
\begin{displaymath}
|S_{\textup{good}}| \lesssim \mu(\lambda Q_0).
\end{displaymath}

\section{The diagonal}
For every $Q \in \mathcal{D}^T$ and $R \in \mathcal{D}^{T^*}$ we write $Q = \bigcup_{i=1}^{2^n} Q_i$ and $R = \bigcup_{j=1}^{2^n} R_j$, where $Q_i \in \textup{ch}(Q)$ and $R_j \in \textup{ch}(R)$.
We then fix two indices $i$ and $j$. We write $Q \sim R$ to mean $\ell(Q) \sim \ell(R)$ and $d(Q,R) \lesssim \min(\ell(Q), \ell(R))$. Notice that $\#\{Q:\, Q \sim R\} \lesssim 1$. We want to bound a sum of the form
\begin{displaymath}
S = \sum_R \sum_{Q: \, Q \sim R} A_{Q,i} |\langle T(1_{Q_i} u_{Q,i}), 1_{R_j} v_{R,j}\rangle| B_{R, j},
\end{displaymath}
where $A_{Q,i} = A_{Q,i}(f), B_{Q,j} = B_{Q,j}(g) \ge 0$ are constants and $u_{Q,i}, v_{R,j}$ are functions such that
\begin{align}
\label{eq:Au} \sum_Q \big[\|1_{Q_i} M_{\mu} u_{Q,i}\|_{L^2(\mu)}^2 + \|1_{Q_i} T u_{Q,i}\|_{L^2(\mu)}^2\big] A_{Q,i}^2 & \lesssim \mu(\lambda Q_0) \\
\label{eq:Bv} \sum_R \big[\|1_{R_j} M_{\mu}v_{R,j}\|_{L^2(\mu)}^2 + \|1_{R_j} T^*v_{R,j}\|_{L^2(\mu)}^2\big] B_{R,j}^2 &\lesssim \mu(\lambda Q_0).
\end{align}
Here we recall that $M_{\mu}$ is the centred maximal function with respect to the measure $\mu$.

In practice, we shall consider $S$ with the choice that $(A_{Q,i}, u_{Q,i})$ is either
\begin{displaymath}
A_{Q,i} = \left\{ \begin{array}{ll} \Big|\frac{\langle f\rangle_{Q_i}}{ \langle b^T_{Q_i^a}\rangle_{Q_i} }- \frac{\langle f\rangle_{Q}}{ \langle b^T_{Q^a}\rangle_{Q} }\Big| & \textrm{if } Q_i^a = Q^a \\
0 & \textrm{if } Q_i^a = Q_i \end{array} \right., \,\,\, u_{Q,i} = b^T_{Q_i^a}
\end{displaymath}
or
\begin{displaymath}
A_{Q,i} = \left\{ \begin{array}{ll} 0  & \textrm{if } Q_i^a = Q^a \\
\Big|\frac{\langle f\rangle_{Q_i}}{ \langle b^T_{Q_i^a}\rangle_{Q_i} }\Big| & \textrm{if } Q_i^a = Q_i \end{array} \right., \,\,\, u_{Q,i} = b^T_{Q_i^a}
\end{displaymath}
or
\begin{displaymath}
A_{Q,i} = \left\{ \begin{array}{ll} 0  & \textrm{if } Q_i^a = Q^a \\
\Big|\frac{\langle f\rangle_{Q}}{ \langle b^T_{Q^a}\rangle_{Q} }\Big| & \textrm{if } Q_i^a = Q_i \end{array} \right., \,\,\, u_{Q,i} = b^T_{Q^a}.
\end{displaymath}
Analogous choices are made for $(B_{R, j}, v_{R,j})$.This means that we consider nine different sums $S$. But to bound a sum of the form $S$ we shall need only the fact that \eqref{eq:Au} and \eqref{eq:Bv} hold, which is true with all these choices:
\begin{lem}
The inequality \eqref{eq:Au} holds with all the above three choices for $(A_{Q,i}, u_{Q,i})$.
\end{lem}
\begin{proof}
This is proved in exactly the same way as the inequality $\sum_Q \|\Delta_Q f\|_{L^2(\mu)}^2 \lesssim \mu(Q^*)$. The proof only needs the additional fact that we have also done a stopping time
with respect to the properties
\begin{displaymath}
\int_Q |M_{\mu} b^T_Q|^2\,d\mu \lesssim \mu(Q) \qquad \textup{and} \qquad \int_Q |T b^T_Q|^2\,d\mu \lesssim \mu(Q).
\end{displaymath}
\end{proof}

\subsection{First surgery: the $\theta$-surgery}
Suppose for convenience that $\ell(Q_i) \le \ell(R_j)$. Let $\theta$ be a small parameter. We perform surgery on $(Q_i, R_j)$ with the parameter $\theta$.
Let $j(\theta)\in\Z$ be such that $2^{-21}\theta\leq 2^{j(\theta)}<2^{-20}\theta$. Let $\mathcal{D}^*$ be yet another random grid in $\R^n$, independent of all other grids considered. Let $G:=\{g\in\mathcal{D}^*:\ell(g)=2^{j(\theta)}\ell(Q_i)\}$, and for $x\in \R^n$, let $G(x)$ be the unique cube in $G$ that contains $x$. We define
\begin{equation*}
  Q_{i,\partial}:=\{x\in Q_i:d(G(x),\partial R_j)<\theta\ell(R_j)/2\}\cup\{x\in Q_i\cap R_j: d(x,\partial G(x))<\theta\ell(G(x))\}.
\end{equation*}
Thus points in $Q_{i,\partial}$ belong to $Q_i$, and are either close to the boundary of $R_j$, or to the boundary of the grid $G$. The set $Q_{i,\partial}$ depends on the set $R_j$ as well. However, we have
\begin{equation}\label{eq:Ibad}
\begin{split}
  Q_{i,\partial}\subset Q_{i,\textup{bad}}
  :=Q_i\cap\Big[ &\bigcup_{\substack{R'\in\mathcal{D}^{T^*}\\ 2^{-r}\ell(Q_i)\leq\ell(R')\leq 2^r\ell(Q_i)}}\{x:d(x,\partial R')<\theta\ell(R')\} \\
    &\cup\bigcup_{\substack{g\in\mathcal{D}^*\\ \ell(g)=2^{j(\theta)}\ell(Q_i)}}\{x:d(x,\partial g)<\theta\ell(g)\}\Big]\\ & =: Q_i \cap \Big[\bigcup_{a=\textup{gen}(Q_i)-r}^{\textup{gen}(Q_i)+r} \partial^{D^{T^*}}_a(\theta)
    \cup \partial^{D^*}_{\textup{gen}(Q_i)-j(\theta)}(\theta)\Big] \\
    &=: Q_i \cap [\partial^{D^{T^*}}_{\sim \textup{gen}(Q_i)}(\theta) \cup \partial^{D^*}_{\textup{gen}(Q_i)-j(\theta)}(\theta)],
\end{split}
\end{equation}
which depends only on $Q_i$ and the grids $\mathcal{D}^{T^*}$ and $\mathcal{D}^*$. One should keep in mind that in what follows $Q_{i,\textup{bad}} = Q_{i,\textup{bad}}(\theta)$.

We set
\begin{equation*}
  Q_{i,\textup{sep}}:=Q_i\setminus (Q_{i,\partial}\cup R_j),
\end{equation*}
the part of $Q_i$ strictly separated from $R_j$. Finally, we have
\begin{equation*}
  Q_{i,\Delta} := Q_i\setminus(Q_{i,\partial}\cup Q_{i,\textup{sep}})
  =\bigcup_k L_k,
\end{equation*}
where each $L_k$ is of the form $L_k=(1-\theta)g\cap Q_i\cap R_j$ for some $g\in G$, and $\#k \lesssim_{\theta} 1$.
In fact, $L_k$ is of the form $L_k=(1-\theta)g$ unless it is close to the boundary of $Q_i$; it cannot be close to the boundary of $R_j$, since such cubes were already subtracted in the $Q_{i,\partial}$ component.

We have the partition
\begin{equation*}
  Q_i=Q_{i,\textup{sep}}\cup Q_{i,\partial}\cup Q_{i,\Delta} = Q_{i,\textup{sep}}\cup Q_{i,\partial}\cup \bigcup_k L_k,
\end{equation*}
and in a completely analogous manner also
\begin{equation*}
  R_j=R_{j,\textup{sep}}\cup R_{j,\partial}\cup R_{j, \Delta} =  R_{j,\textup{sep}}\cup R_{j,\partial} \cup  \bigcup_s L_s.
\end{equation*}
A key observation is that all $L_k\subset Q_i\cap R_j$ appearing in the first union are cubes (of the form $(1-\theta)g$ for $g\in G$) unless they are close to $\partial Q_i$, and they are never close to $\partial R_j$, while the $L_s$ in the second union are cubes unless they are close to $\partial R_j$, and they are never close to $\partial Q_i$. Thus, all $L_k=L_s$ that appear in both unions are cubes and then $5L_k \subset Q_i \cap R_j$.

\subsection{Reduction to a deeply diagonal term}
Using the above $\theta$-surgery we want to reduce to a term of the form $\langle T(1_H u_{Q,i}), 1_H v_{R,j}\rangle$, where $H = L_k = L_s$ is a cube with $5H \subset Q_i \cap R_j$ and $\#H \le C(\theta)$.
This term will then be split using a different $\sigma$-surgery (at the end one will first choose $\theta$ small, and then $\sigma = \sigma(\theta)$ small depending on $\theta$). But let us first do the actual reduction.
 
We write
\begin{align*}
\langle T(1_{Q_i} u_{Q,i}), 1_{R_j} v_{R,j}\rangle &=  \sum_{\beta \in \{\textup{sep},\,\partial\}} \langle T(1_{Q_i} u_{Q,i}), 1_{R_{j,\beta}} v_{R,j}\rangle \\
&+ \sum_{\alpha \in \{\textup{sep},\,\partial\}}   \langle T(1_{Q_{i,\alpha}} u_{Q,i}), 1_{R_{j,\Delta}} v_{R,j}\rangle \\
&+ \sum_{k\ne s} \langle T(1_{L_k} u_{Q,i}), 1_{L_s} v_{R,j}\rangle \\
&+ \sum_{k = s} \langle T(1_{L_k} u_{Q,i}), 1_{L_k} v_{R,j}\rangle.
\end{align*}
If $\alpha = \textup{sep}$ or $\beta = \textup{sep}$ or $k \ne s$, then the corresponding pairing is seen to be dominated by
\begin{displaymath}
C(\theta) \|1_{Q_i} u_{Q,i}\|_{L^2(\mu)} \|1_{R_j} v_{R,j}\|_{L^2(\mu)}
\end{displaymath}
using the size estimate of the kernel $K$ together with the fact that the sets are separated by $c(\theta)\ell(Q_i) \sim c(\theta)\ell(R_j)$.
In the case $k \ne s$ a further large dependence on $\theta$ is gained from the summation $\sum_{k \ne s} 1$.

The sum of the cases $\alpha = \partial$ and $\beta = \partial$ is dominated by
\begin{displaymath}
\|T\|( \|1_{Q_{i,\textup{bad}}} u_{Q,i}\|_{L^2(\mu)} \|1_{R_j} v_{R,j}\|_{L^2(\mu)} + \|1_{Q_i} u_{Q,i}\|_{L^2(\mu)} \|1_{R_{j,\textup{bad}}} v_{R,j}\|_{L^2(\mu)}].
\end{displaymath}
All in all, we have the estimate
\begin{align*}
|\langle T(1_{Q_i} u_{Q,i}), 1_{R_j} v_{R,j}\rangle| &\le C(\theta) \|1_{Q_i} u_{Q,i}\|_{L^2(\mu)} \|1_{R_j} v_{R,j}\|_{L^2(\mu)} \\
&+ \|T\| \|1_{Q_{i,\textup{bad}}} u_{Q,i}\|_{L^2(\mu)} \|1_{R_j} v_{R,j}\|_{L^2(\mu)} \\
&+ \|T\| \|1_{Q_i} u_{Q,i}\|_{L^2(\mu)} \|1_{R_{j,\textup{bad}}} v_{R,j}\|_{L^2(\mu)} \\
&+ \sum_{H := L_k = L_s} |\langle T(1_{H} u_{Q,i}), 1_{H} v_{R,j}\rangle|.
\end{align*}
We will now fix one such $H$ and estimate $|\langle T(1_{H} u_{Q,i}), 1_{H} v_{R,j}\rangle|$ with a bound independent of $H$.

\subsection{Second surgery: the $\sigma$-surgery}
We continue to split
\begin{align*}
\langle T(1_{H} u_{Q,i}), 1_{H} v_{R,j}\rangle &= \langle Tu_{Q, i}, 1_H v_{R, j}\rangle \\
&- \langle T(1_{\R^n \setminus 5H} u_{Q, i}), 1_H v_{R, j}\rangle\\
&- \langle T(1_{5H \setminus (1+\sigma)H} u_{Q, i}), 1_H v_{R, j}\rangle \\
&- \langle T(1_{(1+\sigma)H \setminus H} u_{Q, i}), 1_H v_{R, j}\rangle.
\end{align*}
We have that
\begin{displaymath}
|\langle Tu_{Q, i}, 1_H v_{R, j}\rangle| \le \|1_{Q_i} Tu_{Q, i}\|_{L^2(\mu)} \|1_{R_j} v_{R, j}\|_{L^2(\mu)}.
\end{displaymath}
Using separation and the fact that $5H \subset Q_i \cap R_j$ we see that
\begin{displaymath}
|\langle T(1_{5H \setminus (1+\sigma)H} u_{Q, i}), 1_H v_{R, j}\rangle| \le C(\sigma) \|1_{Q_i} u_{Q, i}\|_{L^2(\mu)} \|1_{R_j} v_{R, j}\|_{L^2(\mu)}.
\end{displaymath}
Moreover, there holds that
\begin{displaymath}
|\langle T(1_{(1+\sigma)H \setminus H} u_{Q, i}), 1_H v_{R, j}\rangle| \le \|T\| \|1_{ \partial^{D^*}_{\textup{gen}(Q_i)-j(\theta)}(\sigma)} 1_{Q_i} u_{Q, i}\|_{L^2(\mu)} \|1_{R_j} v_{R, j}\|_{L^2(\mu)}.
\end{displaymath}

The term $\langle T(1_{\R^n \setminus 5H} u_{Q, i}), 1_H v_{R, j}\rangle$ still requires further splitting. We write
\begin{align*}
\langle T(1_{\R^n \setminus 5H} u_{Q, i}), 1_H v_{R, j}\rangle &= \langle T(1_{\R^n \setminus 5H} u_{Q, i}) - \langle b^{T^*}_H/\mu(H), T(1_{\R^n \setminus 5H} u_{Q, i}) \rangle, 1_H v_{R, j}\rangle \\
&+ \langle b^{T^*}_H, T(1_{\R^n \setminus 5H} u_{Q, i}) \rangle \frac{1}{\mu(H)} \int_H v_{R, j}\,d\mu.
\end{align*}
\begin{lem}
There holds that
\begin{align*}
|\langle T(1_{\R^n \setminus 5H} u_{Q, i}) - \langle b^{T^*}_H/\mu(H), T(1_{\R^n \setminus 5H}& u_{Q, i}) \rangle, 1_H v_{R, j}\rangle| \\  &\lesssim \|1_{Q_i} M_{\mu} u_{Q,i}\|_{L^2(\mu)}  \|1_{R_j} v_{R, j}\|_{L^2(\mu)}.
\end{align*}
\end{lem}
\begin{proof}
Let $\Phi := T(1_{\R^n \setminus 5H} u_{Q, i})$.
We need to bound
\begin{displaymath}
\int |1_H(x)v_{R,j}(x)||\Phi (x) - \langle b^{T^*}_H/\mu(H), \Phi \rangle|\,d\mu(x).
\end{displaymath}
For $x \in H$ we have that
\begin{align*}
|\Phi (x) - \langle b^{T^*}_H/\mu(H), \Phi \rangle| \le \frac{1}{\mu(H)} \int_H |b^{T^*}_H(y)| |\Psi(x) - \Psi(y)|\,d\mu(y).
\end{align*}
But for $x, y \in H$ there holds that
\begin{displaymath}
|\Psi(x) - \Psi(y)| \lesssim \ell(H)^{\alpha} \int_{|x-z| \ge c\ell(H)} \frac{|u_{Q,i}(z)|}{|x-z|^{m+\alpha}}\,d\mu(z) \lesssim M_{\mu} u_{Q,i}(x).
\end{displaymath}
Therefore, for $x \in H$ we have that
\begin{displaymath}
|\Phi (x) - \langle b^{T^*}_H/\mu(H), \Phi \rangle| \lesssim M_{\mu} u_{Q,i}(x)
\end{displaymath}
using which we see that
\begin{align*}
\int |1_H(x)v_{R,j}(x)|&|\Phi (x) - \langle b^{T^*}_H/\mu(H), \Phi \rangle|\,d\mu(x) \\ &\lesssim \int |1_H(x)v_{R,j}(x)||1_H(x)M_{\mu} u_{Q,i}(x) |\,d\mu(x)\\ &\lesssim \|1_{Q_i} M_{\mu} u_{Q,i}\|_{L^2(\mu)}  \|1_{R_j} v_{R, j}\|_{L^2(\mu)}.
\end{align*}
\end{proof}
We are left to deal with
\begin{align*}
\Big|  \langle b^{T^*}_H, T(1_{\R^n \setminus 5H} u_{Q, i}) \rangle \frac{1}{\mu(H)}& \int_H v_{R, j}\,d\mu \Big| \\&\le |\langle b^{T^*}_H, T(1_{\R^n \setminus 5H} u_{Q, i}) \rangle| \mu(H)^{-1/2}  \|1_{R_j} v_{R, j}\|_{L^2(\mu)}.
\end{align*}
Our final splitting is as follows:
\begin{align*}
\langle b^{T^*}_H, T(1_{\R^n \setminus 5H} u_{Q, i}) \rangle &= \langle b^{T^*}_H, Tu_{Q,i}\rangle 
 - \langle b^{T^*}_H, T(1_{5H \setminus (1+\sigma)H} u_{Q,i})\rangle \\
&\qquad- \langle b^{T^*}_H, T(1_{(1+\sigma)H \setminus H} u_{Q,i})\rangle 
 - \langle b^{T^*}_H, T(1_H u_{Q,i})\rangle.
\end{align*}
We have that
\begin{displaymath}
|\langle b^{T^*}_H, Tu_{Q,i}\rangle| \le \|b^{T^*}_H\|_{L^2(\mu)} \|1_{Q_i}Tu_{Q,i}\|_{L^2(\mu)} \lesssim \mu(H)^{1/2}  \|1_{Q_i}Tu_{Q,i}\|_{L^2(\mu)}.
\end{displaymath}
Using again separation and the fact that $5H \subset Q_i \cap R_j$ we see that
\begin{align*}
|\langle b^{T^*}_H, T(1_{5H \setminus (1+\sigma)H} u_{Q,i})\rangle| &\le C(\sigma) \|b^{T^*}_H\|_{L^2(\mu)}\|1_{Q_i}u_{Q,i}\|_{L^2(\mu)}\\ &\lesssim C(\sigma) \mu(H)^{1/2} \|1_{Q_i}u_{Q,i}\|_{L^2(\mu)}.
\end{align*}
Next, notice that
\begin{align*}
|\langle b^{T^*}_H, T(1_{(1+\sigma)H \setminus H} u_{Q,i})\rangle| &\le  \|T\| \|b^{T^*}_H\|_{L^2(\mu)} \|1_{ \partial^{D^*}_{\textup{gen}(Q_i)-j(\theta)}(\sigma)} 1_{Q_i} u_{Q, i}\|_{L^2(\mu)} \\
&\lesssim \|T\| \mu(H)^{1/2} \|1_{ \partial^{D^*}_{\textup{gen}(Q_i)-j(\theta)}(\sigma)} 1_{Q_i} u_{Q, i}\|_{L^2(\mu)}.
\end{align*}
Finally, we have that
\begin{align*}
|\langle b^{T^*}_H, T(1_H u_{Q,i})\rangle| = |\langle T^*b^{T^*}_H, 1_H u_{Q,i}\rangle| &\le \|1_H T^*b^{T^*}_H\|_{L^2(\mu)} \|1_{Q_i} u_{Q,i}\|_{L^2(\mu)} \\
&\lesssim \mu(H)^{1/2}\|1_{Q_i} u_{Q,i}\|_{L^2(\mu)}.
\end{align*}

Collecting the estimates we see that our $\sigma$-surgery yields the final bound
\begin{align*}
|\langle T(1_{H} u_{Q,i}), 1_{H} v_{R,j}\rangle| 
&\le C(\sigma) \|1_{Q_i} u_{Q,i}\|_{L^2(\mu)}\|1_{R_j} v_{R, j}\|_{L^2(\mu)} \\
&+ C \|1_{Q_i} M_{\mu} u_{Q,i}\|_{L^2(\mu)}  \|1_{R_j} v_{R, j}\|_{L^2(\mu)} \\
&+ C  \|1_{Q_i}Tu_{Q,i}\|_{L^2(\mu)} \|1_{R_j} v_{R, j}\|_{L^2(\mu)} \\
&+ C\|T\| \|1_{ \partial^{D^*}_{\textup{gen}(Q_i)-j(\theta)}(\sigma)} 1_{Q_i} u_{Q, i}\|_{L^2(\mu)} \|1_{R_j} v_{R, j}\|_{L^2(\mu)}.
\end{align*}

\subsection{The final estimate through averaging}
Combining the different surgeries we see that
\begin{align*}
|\langle T(1_{Q_i} u_{Q,i}), 1_{R_j} v_{R,j}\rangle| &\le C(\theta, \sigma) \|1_{Q_i} u_{Q,i}\|_{L^2(\mu)} \|1_{R_j} v_{R,j}\|_{L^2(\mu)} \\
&+ \|T\| \|1_{Q_{i,\textup{bad}}} u_{Q,i}\|_{L^2(\mu)} \|1_{R_j} v_{R,j}\|_{L^2(\mu)} \\
&+ \|T\| \|1_{Q_i} u_{Q,i}\|_{L^2(\mu)} \|1_{R_{j,\textup{bad}}} v_{R,j}\|_{L^2(\mu)} \\
&+ C(\theta)  \|1_{Q_i} M_{\mu} u_{Q,i}\|_{L^2(\mu)}  \|1_{R_j} v_{R, j}\|_{L^2(\mu)} \\
&+ C(\theta)   \|1_{Q_i}Tu_{Q,i}\|_{L^2(\mu)} \|1_{R_j} v_{R, j}\|_{L^2(\mu)} \\
&+ C(\theta) \|T\| \|1_{ \partial^{D^*}_{\textup{gen}(Q_i)-j(\theta)}(\sigma)} 1_{Q_i} u_{Q, i}\|_{L^2(\mu)} \|1_{R_j} v_{R, j}\|_{L^2(\mu)}.
\end{align*}
Using the Cauchy--Schwarz inequality, the property $\#\{Q:\, Q \sim R\} \lesssim 1$, the inequalities \eqref{eq:Au} and \eqref{eq:Bv},
the independence of the different dyadic grids $\mathcal{D}^T$, $\mathcal{D}^{T^*}$ and $\mathcal{D}^*$, and the fact that e.g.
$\mathbb{P}_{\mathcal{D^*}}(x \in  \partial^{D^*}_k(\sigma)) \le c(\sigma)$ for every point $x$ and every generation $k \in \Z$, we see that
\begin{align*}
E_{\mathcal{D}^T} E_{\mathcal{D}^{T^*}} E_{\mathcal{D}^*}  \sum_R \sum_{Q: \, Q \sim R} A_{Q,i}& |\langle T(1_{Q_i} u_{Q,i}), 1_{R_j} v_{R,j}\rangle| B_{R, j} \\
&\le [C(\theta, \sigma) + c(\theta)\|T\| + C(\theta)c(\sigma)\|T\|]\mu(\lambda Q_0).
\end{align*}
Here $\lim_{p \to 0} c(p) = 0$. Let $\upsilon > 0$ be small enough. First choose $\theta$ so small that $c(\theta) \le \upsilon$. Then choose $\sigma$ so small that $C(\theta)c(\sigma) \le \upsilon$.
We have proved the diagonal bound
\begin{displaymath}
E_{\mathcal{D}^T} E_{\mathcal{D}^{T^*}}  \sum_R \sum_{Q: \, Q \sim R} A_{Q,i} |\langle T(1_{Q_i} u_{Q,i}), 1_{R_j} v_{R,j}\rangle| B_{R, j} \le [C + 2\upsilon \|T\|]\mu(\lambda Q_0).
\end{displaymath}
We have completed the proof of our main theorem, Theorem \ref{thm:main}.

\end{document}